\documentclass[12pt,reqno]{amsart}

\usepackage{amssymb}

\usepackage[all]{xy}

\numberwithin{equation}{section}

\newtheorem{theorem}{Theorem}[section]
\newtheorem{proposition}[theorem]{Proposition}
\newtheorem{lemma}[theorem]{Lemma}

\theoremstyle{definition}
\newtheorem{definition}[theorem]{Definition}
\newtheorem{remark}[theorem]{Remark}

\begin{document}

\baselineskip=15pt

\title[Opers on transversely holomorphic foliations]{Opers on transversely
holomorphic foliations}

\author[I. Biswas]{Indranil Biswas}

\address{Department of Mathematics, Shiv Nadar University, NH91, Tehsil
Dadri, Greater Noida, Uttar Pradesh 201314, India}

\email{indranil.biswas@snu.edu.in, indranil29@gmail.com}

\author[S. Dumitrescu]{Sorin Dumitrescu}

\address{Universit\'e C\^ote d'Azur, CNRS, LJAD, France}

\email{dumitres@unice.fr}

\subjclass[2010]{}


\date{}

\begin{abstract}
For real codimension two $C^\infty$ foliations, we study the transversely complex 
structures, transversely complex projective structures and transverse opers associated
to the foliation. We prove 
a uniqueness theorem for the transversely holomorphic ${\mathbb 
C}{\mathbb P}^1$--bundle naturally associated with a transversely complex projective 
structures. A similar uniqueness theorem is proved for the transversely holomorphic
filtered ${\mathbb C}{\mathbb P}^{r-1}$--bundle naturally associated to a transverse
${\rm PGL}(r,{\mathbb C})$--oper.
\end{abstract}

\maketitle

\tableofcontents

\section{Introduction}

The study of second-order linear differential equations, and of the associated Riccati 
equations, on Riemann surfaces is intimately related to the geometry of {\it complex 
projective structures} (i.e., an atlas with transition functions in ${\rm PGL}(2,{\mathbb 
C})$) and provides a major tool for a strategy leading to a proof of the uniformization 
theorem for Riemann surfaces \cite{Po}; for a modern treatment see, for instance, 
\cite{Gu1,Gu2}, \cite{LM} and \cite[Chapter 9]{StG}.

This classical correspondence admits a natural reformulation in the modern language of 
{\it opers} introduced by Beilinson and Drinfeld \cite{BD1,BD2} (building on the seminal 
work of Drinfeld and Sokolov \cite{DS1, DS2}): rank $r$ opers provide an intrinsic 
geometric framework for the study of linear differential equations of order $r$. Roughly 
speaking, an oper is a connection on a filtered bundle satisfying a transversality 
condition with respect to the filtration. In rank two, ${\rm PGL}(2,{\mathbb C})$-opers 
can be identified with the above mentioned complex projective structures.

The concept of oper was found to be very important, not only in the study of ordinary 
differential equations, but also in many other topics as geometric Langlands 
correspondence, nonabelian Hodge theory and also in some branches of mathematical 
physics; see, for example, \cite{BF}, \cite{DFKMMN}, \cite{FT}, \cite{FG}, \cite{FG2}, 
\cite{CS}, \cite{Fr}, \cite{Fr2}, \cite{BSY} and references therein. Nowadays the study 
of opers have been firmly established as an important topic
 in mathematics and mathematical physics. For instance, important progress in the 
understanding of opers was worked out in \cite{BD1, BD2, FG, FG2, AB, W, ABF}.

We introduce and study here {\it transverse ${\rm PGL}(r,{\mathbb C})$-opers} for a real 
codimension two foliation on a smooth manifold. We start by defining {\it transversely 
complex structures} and then compatible {\it transversely complex projective structures}.
The transverse direction being of complex dimension one, a transverse complex 
projective structure can be seen as the geometric manifestation of a second-order 
differential equation in the transverse variable, preserved by the holonomy of the 
foliation. This naturally generalizes to give rise to the notion of a {\it transverse 
oper} (it introduce in Definition \ref{def3}) which provides a coordinate-free 
framework for encoding the transversely complex differential equation in terms of a 
filtered flat connection satisfying a transversality condition with respect to the 
filtration.
 
Notice that the global classification of codimension one holomorphic foliations with 
transverse complex projective structures was worked out by several authors (see, for 
instance, \cite{LP,LPT,Sc} and references therein). It was recently proved in 
\cite{BD,FMP} that, in general, holomorphic foliations on compact complex K\"ahler 
surfaces do not admit compatible transverse complex projective structures.
 
Our study here of transverse opers is local. First we characterize real codimension two 
foliations with a transversely complex projective structure in terms of an Ehresmann type 
data, namely a transversely principal ${\rm PGL}(2,{\mathbb C})$--bundle endowed with a 
transversely holomorphic structure and a compatible flat connection together with a 
holomorphic section of the associated transversely holomorphic ${\mathbb C}{\mathbb 
P}^1$--bundle satisfying a transversality condition with respect to the induced 
connection (see Theorem \ref{thm0} here and compare with \cite{LP} and with \cite[Lemma 
4.3]{BD}). One main result of our paper proves that the associated transversely 
holomorphic ${\mathbb C}{\mathbb P}^1$--bundle is independent of the transverse complex 
projective structure. More precisely, Theorem \ref{thm1} asserts that the above 
transversely holomorphic ${\mathbb C}{\mathbb P}^1$--bundle is isomorphic to the 
projectivization of the transversely holomorphic jet bundle of order 1 (which is uniquely 
determined by the transversely complex structure) and its transversely holomorphic section 
is given by a natural section --- obtained from the exact sequence of jet bundles -- of
this transverse jet bundle of order one. This also 
enables us to parametrize the space of transversely complex projective structures to a 
given transversely complex foliation (see Proposition \ref{prop1} which
generalizes \cite[Lemma 4.1]{BD}).

A same uniqueness result is proved here in the case of transverse ${\rm PGL}(r,{\mathbb 
C})$-opers. In this case we prove that the associated transversely holomorphic filtered 
${\mathbb C}{\mathbb P}^{r-1}$--bundle is independent of the oper and it is canonically 
isomorphic to the projectivization of a transversely holomorphic jet bundle of order 
$r-1$ (see Theorem \ref{thm2}). Our proof establish the compatibility of the above 
transversely holomorphic filtered ${\mathbb C}{\mathbb P}^{r-1}$--bundle with the 
transversely holomorphic jet bundle of order $r-1$ endowed with its natural filtration 
by lower order transverse jet bundles.

\section{Foliations and transversely projective structures}

\subsection{Partial connections}

Let $M$ be a connected $C^\infty$ manifold of dimension $d+2$, where $d$ is a 
nonnegative integer. Let
$$
{\mathcal F}\ \subset\ T^{\mathbb R}M
$$
be a $C^\infty$ subbundle, of rank $d$, of the real tangent bundle $T^{\mathbb R}M$. Assume that
${\mathcal F}$ is closed under the Lie bracket operation on the sheaf of $C^\infty$ vector
fields on $M$. In other words, the distribution ${\mathcal F}$ is integrable, and it
defines a $C^\infty$ nonsingular foliation on $M$ of dimension $d$. Let
\begin{equation}\label{e1}
\phi\ : \ (T^{\mathbb R}M)^* \ \longrightarrow\ {\mathcal F}^*
\end{equation}
be the dual of the inclusion map ${\mathcal F}\,\hookrightarrow\, T^{\mathbb R}M$. The leaf-wise
exterior derivation produces a homomorphism of sheaves
\begin{equation}\label{d}
d_{\mathcal F} \ :\ \bigwedge\nolimits^i {\mathcal F}^* \ \longrightarrow\
\bigwedge\nolimits^{i+1} {\mathcal F}^*
\end{equation}
for all $i\, \geq\, 0$. The following is a commutative diagram of homomorphisms of sheaves:
$$
\begin{matrix}
\bigwedge\nolimits^i (T^{\mathbb R}M)^* & \xrightarrow{\,\,\, \wedge^i \phi\,\,\,}&
\bigwedge\nolimits^i {\mathcal F}^*\\
\,\, \Big\downarrow d && \,\,\,\, \Big\downarrow d_{\mathcal F}\\
\bigwedge\nolimits^{i+1} (T^{\mathbb R}M)^* & \xrightarrow{\,\,\, \wedge^{i+1} \phi\,\,\,}&
\bigwedge\nolimits^{i+1} {\mathcal F}^*
\end{matrix}
$$
where $\bigwedge^j \phi\, :\, \bigwedge\nolimits^j (T^{\mathbb R}M)^*\,\longrightarrow\,
\bigwedge\nolimits^{j+1} (T^{\mathbb R}M)^* $ is the homomorphism induced by
the surjective map $\phi$ in \eqref{e1}. In particular, we have $d_{\mathcal F}(f)\,=\,
\phi(df)$ for all locally defined $C^\infty$ functions $f$.

Take a $C^\infty$ vector bundle $V$, which may be a real or complex one, on $M$. A \textit{partial
connection} on $V$ is a $C^\infty$ differential operator of first order
$$
D \ : \ V \ \longrightarrow\ V\otimes {\mathcal F}^*
$$
such that $D(fs) \,=\, f\cdot D(s)+ s\otimes d_{\mathcal F}(f)$ for all locally defined $C^\infty$
functions $f$ on $M$ and all locally defined $C^\infty$ sections $s$ of $V$. Denote
by ${\mathcal K}(D)$ the following composition of differential operators:
$$
V \ \xrightarrow{\,\,\, D\,\,\,} \ V\otimes {\mathcal F}^* \ \xrightarrow{\,\,\,
\widetilde{D\otimes{\rm Id}_{{\mathcal F}^*}}+ \text{Id}_V\otimes d_{\mathcal F} \,\,\,}\
V\otimes \bigwedge\nolimits^2 {\mathcal F}^*,
$$
where $d_{\mathcal F}$ is the homomorphism in \eqref{d} and
$\widetilde{D\otimes{\rm Id}_{{\mathcal F}^*}}$ is the composition of homomorphisms
$D\otimes_{\mathbb R}{\rm Id}_{{\mathcal F}^*}\, :\, V\otimes_{\mathbb R} {\mathcal F}^*\, \longrightarrow\,
V\otimes {\mathcal F}^*\otimes_{\mathbb R} {\mathcal F}^*$ with ${\rm Id}_V\otimes p\, :\,
V\otimes {\mathcal F}^*\otimes_{\mathbb R} {\mathcal F}^*\,\longrightarrow\, 
V\otimes \bigwedge\nolimits^2 {\mathcal F}^*;$ here $p$ is the natural projection
${\mathcal F}^*\otimes {\mathcal F}^*\,\longrightarrow\,\bigwedge\nolimits^2 {\mathcal F}^*$.
It is straightforward to check that
${\mathcal K}(D)(fs)\,=\, f{\mathcal K}(D)(s)$ for all locally defined $C^\infty$ functions
$f$ on $M$ and all locally defined $C^\infty$ sections $s$ of $V$. This implies that
\begin{equation}\label{e2}
{\mathcal K}(D)\ \in\ C^\infty\left(M,\,\,\, \text{End}(V)\otimes \bigwedge\nolimits^2
{\mathcal F}^*\right);
\end{equation}
this section ${\mathcal K}(D)$ is called the curvature of $D$. The partial connection $D$ is
called \textit{integrable}
(or \textit{flat}) if we have ${\mathcal K}(D)\,=\, 0$.

Let
\begin{equation}\label{e3}
N \ :=\ (T^{\mathbb R}M)/{\mathcal F}
\end{equation}
be the normal bundle of the foliation $\mathcal F$. Since we have $[{\mathcal F},\, {\mathcal F}]\, \subset\,
{\mathcal F}$, the Lie bracket operation of vector fields ${\mathcal F}\otimes_{\mathbb R} T^{\mathbb R}M
\, \longrightarrow\, T^{\mathbb R}M$ produces a homomorphism of sheaves
$$\widetilde{D}^N\ :\ {\mathcal F}\otimes_{\mathbb R} (T^{\mathbb R}M/{\mathcal F})\,= \,
{\mathcal F}\otimes_{\mathbb R} N \, \longrightarrow\, (T^{\mathbb R}M/{\mathcal F})\,=N,$$
where $N$ is the normal bundle in \eqref{e3}. Note that
$\widetilde{D}^N ((f\cdot v)\otimes w)\,=\,f\cdot \widetilde{D}^N (v\otimes w)$ for all
locally defined smooth sections $v$ (respectively, $w$) of ${\mathcal F}$ (respectively, $N$) and all
locally defined smooth functions $f$ on $M$. Hence $\widetilde{D}^N$ produces a homomorphism of sheaves
\begin{equation}\label{e4}
D^N\, :\, N\, \longrightarrow\, N\otimes{\mathcal F}^* .
\end{equation}
It is straightforward to check that $D^N$ is a partial connection on $N$ in the direction of
$\mathcal F$. The Jacobi identity for the Lie bracket operation of vector fields on $M$ ensures that the
partial connection $D^N$ in \eqref{e4} is actually integrable.

The partial connection $D^N$ in \eqref{e4} is known as the Bott partial connection.

\subsection{Transversely projective structure}\label{se2.2}

\begin{definition}\label{def1}
A {\it transversely complex structure} on $\mathcal F$ is defined by data
\begin{equation}\label{e5}
\{(U_i,\, \sigma_i)\}_{i\in I}
\end{equation}
such that
\begin{enumerate}
\item{} each $U_i$ is an open subset of $M$ with $\bigcup_{i\in I} U_i \,=\, M$,

\item{} each $\sigma_i$ is a $C^\infty$ submersion of $U_i$ to an open subset $D_i$ of ${\mathbb C}
{\mathbb P}^1$ satisfying the condition that the subbundle ${\mathcal F}\big\vert_{U_i} \, \subset\,
T^{\mathbb R} U_i$ coincides with the kernel of the differential map $d\sigma_i\, :\, TU_i \,
\longrightarrow \, \sigma^*_i TD_i$ for the submersion $\sigma_i$, and

\item{} for every pair $i,\,j\,\in\, I$, and every connected component
$U^c_{i,j}\, \subset\, U_i\cap U_j$, there is a commutative diagram of maps
\begin{equation}\label{e6}
\begin{matrix}
U^c_{i,j} & \stackrel{\rm Id}{\longrightarrow} & U^c_{i,j}\\
\Big\downarrow\vcenter{\rlap{$\sigma_i$}} &&
\Big\downarrow\vcenter{\rlap{$\sigma_j$}}\\
\sigma_i(U^c_{i,j}) & \xrightarrow{\,\,\, f^c_{i,j}\,\,\,} &
\sigma_j(U^c_{i,j})
\end{matrix}
\end{equation}
where $f^c_{i,j}$ is a holomorphic map (see \cite{Go}). Note that
$f^c_{i,j}$ is uniquely determined by the above commutative diagram.
\end{enumerate}
Two such data $\{(U_i,\, \sigma_i)\}_{i\in I}$ and $\{(U_j,\,
\sigma_j)\}_{j\in J}$ are called {\it equivalent} if their union, namely
$$
\{(U_i,\, \sigma_i)\}_{i\in I\cup J},
$$
also satisfies the above three conditions.
A {\it transversely complex} structure on $\mathcal F$ is
an equivalence class of data $\{(U_i,\, \sigma_i)\}_{i\in I}$ of the above type satisfying
the above three conditions.
\end{definition}

To give another description of the transversely complex structures,
let
\begin{equation}\label{ej}
J\, :\, N\, \longrightarrow\, N
\end{equation}
be a $C^\infty$ endomorphism of the vector bundle $N$ in \eqref{e3} such that
\begin{itemize}
\item $J^2\,=\, - {\rm Id}_N$, and

\item $J$ is flat with respect to the partial connection $D^N$ in \eqref{e4}; in other words,
\begin{equation}\label{j}
D^N\circ J \ =\ (J\otimes {\rm Id}_{{\mathcal F}^*})\circ D^N
\end{equation}
(equivalently, $J$ preserves $D^N$).
\end{itemize}
This endomorphism $J$ will be called a \textit{transversely almost complex structure}.

We will describe a natural identification between the transversely complex structures and
the transversely almost complex structures.

Take a transversely almost complex structure $J$ on $M$. Let $\sigma$ be a $C^\infty$ submersion
of a simply connected open subset $U\, \subset\, M$ to a simply connected open subset $D$ of $S^2\,=\, {\mathbb
R}^2 \cup \{\infty\}$ satisfying the condition that the subbundle ${\mathcal F}\big\vert_{U}
\,\subset\, T^{\mathbb R} U_i$ is the kernel of the differential map $d\sigma\, :\, TU \,
\longrightarrow \, \sigma^* TD$ of $\sigma$. So the differential $d\sigma\, :\, T^{\mathbb R}M\,
\longrightarrow\, \sigma^* T^{\mathbb R} D$ identifies $N$ in \eqref{e3} with $\sigma^*
T^{\mathbb R} D$. Then there is a unique almost complex structure $J_D\, :\, T^{\mathbb R} D
\, \longrightarrow\, T^{\mathbb R} D$ on $D$ such that
$J\big\vert_U\,=\, \sigma^* J_D$ in terms of the above identification of $N$ with $\sigma^*T^{\mathbb R} D$. Indeed, this follows immediately from \eqref{j}. Recall that, by a result of Korn and Lichtenstein almost complex structures on surfaces are integrable (see,
 for instance, \cite[Chapter 1]{StG} and references therein). Identify the Riemann surface
$(D,\, J_D)$ with an open subset $D'$ of ${\mathbb C}{\mathbb P}^1$ (for this identification one can use the uniformization theorem of Riemann surfaces \cite{StG}). Let $\sigma'\, :
\, U\, \longrightarrow\, D'$ be the map given by the composition of $\sigma$
with the above identification of $D$ with $D'$. All such
pairs $(U,\, \sigma')$ together define a transversely complex structure on entire $M$.

Conversely, given a transversely complex structure on $M$ defined by $\{(U_i,\, \sigma_i)\}_{i\in I}$,
identify $N\big\vert_{U_i}$ with $\sigma^*_i T^{\mathbb R} {\mathbb C}{\mathbb P}^1$ using the
differential $d\sigma_i$ of the map $\sigma_i$. Using this identification, the almost complex structure
of ${\mathbb C}{\mathbb P}^1$ pulls back --- by $\sigma_i$ --- to
define a transversely almost complex structure on $U_i$. The condition in \eqref{e6} ensures
that these locally defined transversely almost complex structures patch together
compatibly. Consequently, we obtain a transversely almost complex structure on entire $M$.

The notation $\widehat{\mathcal F}$ will be used for a foliation $\mathcal F$ on $M$
equipped with a transversely complex structure.

\begin{definition}\label{def2}
A {\it transversely projective structure} on $\mathcal F$ is defined by giving data $\{(U_i,\,
\sigma_i)\}_{i\in I}$ as in Definition \ref{def1} satisfying the extra condition (apart from the earlier
three conditions for transversely complex structures) that the holomorphic maps $f^c_{i,j}$ in condition
\eqref{e6} are of the form $z \,\longmapsto\, (az+b)/(cz+d)$, where $a,\,b,\,c,\,d\,\in\, {\mathbb C}$
with $ad-bc \,=\,1$, i.e., each $f^c_{i,j}$ is the restriction of some M\"obius transformation. As
before, two such data $\{(U_i,\, \sigma_i)\}_{i\in I}$ and $\{(U_j,\, \sigma_j)\}_{j\in J}$ are
called {\it equivalent} if their union $\{(U_i,\, \sigma_i)\}_{i\in I\cup J}$ also satisfies all the
conditions for a transversely projective structure.

A {\it transversely projective} structure on $\mathcal F$ is
an equivalence class of such data.
\end{definition}

Clearly, a transversely projective structure on $\mathcal F$ defines a transversely complex 
structure on $\mathcal F$. If $\widehat{\mathcal F}$ is a transversely complex structure on 
$\mathcal F$, then a transversely projective structure on $\widehat{\mathcal F}$ is a 
transversely projective structure $\mathcal P$ on $\mathcal F$ such that the transversely 
complex structure defined by $\mathcal P$ coincides with the transversely complex structure 
$\widehat{\mathcal F}$.

\section{Holomorphic vector bundles and projective connections}\label{se3}

\subsection{Transversely holomorphic structure}\label{se3.1}

Let $\widehat{\mathcal F}$ be a transversely complex structure on the foliation
$\mathcal F$ on $M$. Denote by $J$ the transversely almost complex structure on
$N$ corresponding to $\widehat{\mathcal F}$ (see \eqref{ej}). Using $J$, we
get a decomposition of the complexified normal bundle in \eqref{e3}
\begin{equation}\label{ej2}
N\otimes_{\mathbb R} {\mathbb C}\ =\ N^{1,0}\oplus N^{0,1};
\end{equation}
more precisely, $N^{1,0}$ (respectively,
$N^{0,1}$) is the eigenbundle, for the eigenvalue $\sqrt{-1}$ (respectively, $-\sqrt{-1}$), of the
$\mathbb C$--linear
automorphism $J\otimes_{\mathbb R} {\rm Id}_{\mathbb C}$ of $N\otimes_{\mathbb R} {\mathbb C}$.
Since $J$ is flat with respect the partial connection $D^N$ on $N$ (see \eqref{j}), it follows
immediately that the decomposition in \eqref{ej2} is preserved by the flat
partial connection on $N\otimes_{\mathbb R} {\mathbb C}$ induced by $D^N$. Let
\begin{equation}\label{e7}
{}^N\nabla^{0,1} \ : \ N^{0,1} \ \longrightarrow\, N^{0,1}\otimes {\mathcal F}^*
 \ \ \text{ and }\ \ {}^N\nabla^{1,0} \ : \ N^{1,0} \ \longrightarrow\, N^{1,0}\otimes {\mathcal F}^*
\end{equation}
be the flat partial connections on $N^{0,1}$ and $N^{1,0}$ respectively given by the above mentioned 
flat partial connection on $N\otimes_{\mathbb R} {\mathbb C}$ (which is induced
by $D^N$). If $D$ is a partial connection on a complex vector
bundle $V$ on $M$, then
\begin{equation}\label{e8}
{\mathcal D}\ :=\ D\circ {\rm Id}_{(N^{0,1})^*}+ {\rm Id}_V\otimes \left({}^N\nabla^{0,1}\right)^*
\end{equation}
is a partial connection on $V\otimes \left(N^{0,1}\right)^*$, where 
$\left({}^N\nabla^{0,1}\right)^*$ denotes the partial connection on the dual vector bundle 
$\left(N^{0,1}\right)^*$ induced by the partial connection ${}^N\nabla^{0,1}$ on 
$N^{0,1}$ in \eqref{e7}. The curvature ${\mathcal K}({\mathcal D})$ of the connection
$\mathcal D$ in \eqref{e8} is ${\mathcal K}(D)\otimes {\rm Id}_{(N^{0,1})^*}+
{\rm Id}_V\otimes {\mathcal K}(\left({}^N\nabla^{0,1}\right)^*)$, where
${\mathcal K}(D)$ (respectively, ${\mathcal K}(\left({}^N\nabla^{0,1}\right)^*)$) is
the curvature of $D$ (respectively, $\left({}^N\nabla^{0,1}\right)^*$). Note that
${}^N\nabla^{0,1}$ is flat because $D^N$ is flat, and hence the dual connection
$\left({}^N\nabla^{0,1}\right)^*$ is also flat. Thus if the partial connection $D$ on $V$ is
flat, then the partial connection $\mathcal D$ in \eqref{e8} is also flat.

A \textit{transversely complex vector bundle} on ${\mathcal F}$ is a pair of the form $(V,\, D)$, where $V$
is a $C^\infty$ complex vector bundle on $M$ and $D$ is a flat partial connection on $V$ along $\mathcal F$.
Given transversely complex vector bundles $(V,\, D)$ and $(W,\, D')$, a homomorphism 
$(V,\, D)\, \longrightarrow\, (W,\, D')$ between them is a $C^\infty$ homomorphism of vector bundles
$\rho\,:\, V\, \longrightarrow\, W$ such that $D'\circ\rho\,=\, (\rho\otimes{\rm Id}_{{\mathcal F}^*})
\circ D$ (in other words, the homomorphism $\rho$ intertwines the partial connections $D$ and $D'$).

In order to define transversely holomorphic vector bundles on $\widehat{\mathcal F}$, a certain complex
foliation on $M$ needs to be constructed.

Consider the complexified tangent bundle $T^{\mathbb C}M \, :=\,
T^{\mathbb R}M\bigotimes_{\mathbb R}\mathbb C$. The complexification of the quotient map
$$
q\ :\ T^{\mathbb R}M \ \longrightarrow\ T^{\mathbb R}M/{\mathcal F}\ =:\ N
$$
(see \eqref{e3}) is denoted by
\begin{equation}\label{d2}
q_{\mathbb C}\ :\ T^{\mathbb C}M \ \longrightarrow\ N \otimes_{\mathbb R}\mathbb C.
\end{equation}
Composing it with the natural projection $q^{1,0}\, :\, N \otimes_{\mathbb R}\mathbb C\,
\longrightarrow\, N^{1,0}$
(corresponding to the decomposition in \eqref{ej2}) we get a homomorphism of complex
vector bundles
\begin{equation}\label{e8b}
\widetilde{q} \ :=\ q^{1,0}\circ q_{\mathbb C}\ :\ T^{\mathbb C}M \ \longrightarrow\ N^{1,0}.
\end{equation}
Note that $\widetilde{q}$ is surjective because both $q_{\mathbb C}$ and $q^{1,0}$ are
surjective. Let
\begin{equation}\label{e8c}
{\mathbb F} \ :=\ \text{kernel}(\widetilde{q}) \ \subset\ T^{\mathbb C}M
\end{equation}
be the complex subbundle of $T^{\mathbb C}M$ given by the kernel of $\widetilde{q}$ in \eqref{e8b}. From
the given conditions that the foliation $\mathcal F$ is integrable, and $J$ is flat with respect
to $D^N$ (see \eqref{j}), it follows immediately that ${\mathbb F}$ is
closed under the Lie bracket operation on the sheaf of $C^\infty$
sections of $T^{\mathbb C}M$ (in other words, ${\mathbb F}$ is also integrable).

Note that ${\mathcal F}\, \subset\, \mathbb F$; in fact, we have
\begin{equation}\label{ei}
{\mathcal F}\otimes_{\mathbb R}
{\mathbb C}\ \, \subset\ \, \mathbb F,
\end{equation}
and
\begin{equation}\label{ei2}
{\mathbb F}/({\mathcal F}\otimes_{\mathbb R}{\mathbb C})\ =\ (N\otimes_{\mathbb R}{\mathbb C})/N^{1,0}
\ =\ N^{0,1}.
\end{equation}

Fix a transversely complex structure on ${\mathcal F}$, and denote the resulting 
transversely complex foliation by $\widehat{\mathcal F}$. Let $(V,\, D)$ be a 
transversely complex vector bundle on ${\mathcal F}$. A \textit{transversely 
holomorphic} structure on $(V,\, D)$ is a partial flat connection
\begin{equation}\label{ed}
\mathbb D\ :\ V \ \longrightarrow\ V\otimes {\mathbb F}^*
\end{equation}
on $V$ 
along $\mathbb F$ (see \eqref{e8c}) such that the restriction of $\mathbb D$ to the 
subbundle ${\mathcal F}\, \subset\, \mathbb F$ (see \eqref{ei}) coincides with $D$. A 
\textit{transversely holomorphic} vector bundle on $\widehat{\mathcal F}$ is a 
transversely complex vector bundle on ${\mathcal F}$ equipped with a transversely 
holomorphic structure.

For transversely holomorphic vector bundles $(V_1,\, {\mathbb D}_1)$ and $(V_2,\, {\mathbb D}_2)$, a 
transversely holomorphic homomorphism $(V_1,\,{\mathbb D}_1)\, \longrightarrow\, (V_2,\,{\mathbb D}_2)$ 
between them is a $C^\infty$ homomorphism $\rho\, :\, V_1\, \longrightarrow\, V_2$ of vector bundles such 
that ${\mathbb D}_2\circ\rho\,=\, (\rho\otimes{\rm Id}_{{\mathbb F}^*})\circ{\mathbb D}_1$ (in other
words, the homomorphism $\rho$ intertwines the partial connections ${\mathbb D}_1$ and ${\mathbb D}_2$).

Let ${\mathcal C}({\mathcal F})$ denote the sheaf on $M$ given by the locally defined $C^\infty$ complex
valued functions that are locally constant along the leaves of $\mathcal F$. In other words, 
${\mathcal C}({\mathcal F})$ is the kernel of
$$
d_{\mathcal F} \ :\ C^\infty(M) \ \longrightarrow\ {\mathcal F}^*
$$
(set $i\,=\, 0$ in \eqref{d}). Take a transversely complex vector bundle $(V,\, D)$ on $M$ for ${\mathcal F}$.
Let $\mathbb V$ denote the sheaf on $M$ given by the locally defined integrable (same as flat) sections
of $V$ for the partial connection $D$. Similarly, $({\mathbb N}^{0,1})^*$ denotes the sheaf on $M$ given
by the locally defined flat sections of the dual vector bundle $(N^{0,1})^*$ (see \eqref{ej2}) for the
partial connection $({}^N\nabla^{0,1})^*$ on $(N^{0,1})^*$ in \eqref{e8}. Therefore, the tensor product
${\mathbb V}\bigotimes_{{\mathcal C}({\mathcal F})} ({\mathbb N}^{0,1})^*$ is the sheaf on $M$ given by the
locally defined flat sections of $V\otimes (N^{0,1})^*$ for the partial connection
${\mathcal D}$ on $V\otimes (N^{0,1})^*$ constructed in \eqref{e8}.

We will show that giving a transversely holomorphic structure on $(V,\, D)$ is equivalent to giving
a first order differential operator
\begin{equation}\label{e8d}
\overline{\partial}_V \ :\ {\mathbb V} \ \longrightarrow\ {\mathbb V} \otimes_{{\mathcal C}({\mathcal F})}
({\mathbb N}^{0,1})^*
\end{equation}
that satisfies the Leibniz identity; we note that the Leibniz identity says
that for any locally defined section $s$
of $\mathbb V$ and any locally defined complex function $f$ lying in ${\mathcal C}({\mathcal F})$,
$$
\overline{\partial}_V (fs) \ =\ f\cdot \overline{\partial}_V (s)+ s\otimes \widehat{q}(df),
$$
where $\widehat{q}\, :\, (N\otimes_{\mathbb R} {\mathbb C})^*\, =\, (N^{1,0})^*\oplus 
(N^{0,1})^* \, \longrightarrow\, (N^{0,1})^*$ is the natural projection (see 
\eqref{ej2}); note that the complex $1$-form $df$ lies in the image $(N\otimes_{\mathbb 
R} {\mathbb C})^*\, \hookrightarrow\, (T^{\mathbb C}M)^*$ (this inclusion map is the 
dual of the projection $q_{\mathbb C}$ in \eqref{d2}) because $f$ is locally constant 
along the leaves of $\mathcal F$.

To see the equivalence between the above two descriptions of a transversely holomorphic 
structure on $(V,\, D)$, let $\mathbb D$ be a transversely holomorphic structure on 
$(V,\, D)$ (as in \eqref{ed}). So $\mathbb D$ is a partial flat connection on $V$ along 
$\mathbb F$ (see \eqref{e8c}) such that the restriction of $\mathbb D$ to the subbundle 
${\mathcal F}\, \subset\, \mathbb F$ coincides with $D$. Take a $C^\infty$ section $s\, 
\in\, \Gamma(U,\, V)$ defined over an open subset $U\, \subset\, M$ such that $s$ is 
flat with respect to $D$. Consider ${\mathbb D}(s)\, \in\, \Gamma(U,\, V\otimes {\mathbb 
F}^*)$. The given condition that the restriction of $\mathbb D$ to the subbundle 
${\mathcal F}\, \subset\, \mathbb F$ coincides with $D$ implies that the image of 
${\mathbb D}(s)$ under the natural $\mathbb R$--linear projection
$$
V\otimes {\mathbb F}^* \ \longrightarrow\ V\otimes
({\mathcal F}\otimes_{\mathbb R}{\mathbb C})^*
$$
(this is the dual of the inclusion map in \eqref{ei}) vanishes identically; indeed, recall
that $s$ is flat with respect to $D$. Consequently, we have
\begin{equation}\label{ds}
{\mathbb D}(s)\, \in\, \Gamma(U,\, {\mathbb V} \otimes_{{\mathcal C}
({\mathcal F})} ({\mathbb N}^{0,1})^*)
\end{equation}
(see \eqref{ei2}). The map $s\, \longmapsto\, {\mathbb D}(s)$, where ${\mathbb
D}(s)$ is constructed in \eqref{ds}, clearly satisfies the Leibniz identity.

Conversely, given a differential operator
$$
\overline{\partial}_V \ :\ {\mathbb V} \ \longrightarrow\ 
{\mathbb V} \otimes_{{\mathcal C}({\mathcal F})}({\mathbb N}^{0,1})^*
$$
as in \eqref{e8d} that satisfies the Leibniz identity, consider the homomorphism
${\mathbb V}\, \longrightarrow\, V\otimes {\mathbb F}^*$ obtained by composing $\overline{\partial}_V$
with the natural inclusion map
$$
{\mathbb V} \otimes_{{\mathcal C}({\mathcal F})}({\mathbb N}^{0,1})^*\ \hookrightarrow\,
V\otimes {\mathbb F}^*.
$$
It is straightforward to check that this homomorphism ${\mathbb V}\, \longrightarrow\, V\otimes {\mathbb F}^*$
actually extends uniquely to a flat partial connection $V\, \longrightarrow\, V\otimes {\mathbb F}^*$ which
in fact defines a transversely holomorphic structure on $(V,\, D)$.

The complex line bundle $N^{1,0}$ in \eqref{ej2} equipped with the partial connection ${}^N\nabla^{1,0}$ (see
\eqref{e7}) has a natural transversely holomorphic structure given by the Lie bracket of vector fields on $M$
and the projection of vector fields to $N^{1,0}$.

To give further examples of transversely holomorphic vector bundles, let $W$ be a complex vector
bundle on $M$ equipped with a flat connection $\nabla\, :\, W\, \longrightarrow\, W\otimes
(T^{\mathbb R}M)^*$. Consider the differential operator
$$\nabla^{\mathcal F}\ :\ W\ \longrightarrow\ W\otimes {\mathcal F}^*$$
obtained by composing $\nabla$ with $\text{Id}_W\otimes \phi$, where $\phi$ is the projection
in \eqref{e1}. The pair $(W,\, {\nabla}^{\mathcal F})$ is a transversely complex vector bundle
on $(M,\, {\mathcal F})$.

Let $\iota\, :\, (T^{\mathbb R}M)^*\, \hookrightarrow\, (T^{\mathbb R}M)^*\otimes_{\mathbb R}
{\mathbb C}\,=\, (T^{\mathbb C}M)^*$ be the natural inclusion map defined by $w\, \longmapsto\, w\otimes 1$.
Consider the homomorphism
\begin{equation}\label{cc}
\nabla^{\mathbb C}\ :=\ ({\rm Id}_W\otimes\iota)\circ \nabla \ :\ W\ \longrightarrow\
W\otimes (T^{\mathbb C}M)^*.
\end{equation}
Using the natural projection $(T^{\mathbb C}M)^*\, \longrightarrow\, {\mathbb F}^*$ (it is the
dual of the inclusion map in \eqref{e8c}), the homomorphism $\nabla^{\mathbb C}$ in \eqref{cc}
produces a homomorphism
$$
\nabla^{\mathbb F}\ :\ W\ \longrightarrow\ W\otimes {\mathbb F}^*.
$$
It is straightforward to check that $(W,\, \nabla^{\mathbb F})$ is a transversely holomorphic
structure on the above defined complex vector bundle $(W,\, {\nabla}^{\mathcal F})$.

Given two transversely holomorphic vector bundles $(V,\, D,\, \overline{\partial}_V)$ 
and $(W,\, D',\, \overline{\partial}_W)$, a transversely holomorphic homomorphism $$(V,\ 
D,\ \overline{\partial}_V)\,\ \longrightarrow \,\ (W,\ D',\ \overline{\partial}_W)$$ 
between them is a homomorphism $\rho\, :\, (V,\, D)\, \longrightarrow\, (W,\, D')$ of 
transversely complex vector bundles such that $\rho$ intertwines $\overline{\partial}_V$ 
and $\overline{\partial}_W$, which means that
$$
\overline{\partial}_W\circ\rho\ =\ (\rho\otimes{\rm Id}_{({\mathbb N}^{0,1})^*}) \circ\overline{\partial}_V;
$$
note that $\rho$ produces a homomorphism ${\mathbb V}\, \longrightarrow\, {\mathbb W}$, 
which is also denoted by $\rho$; here $\mathbb W$ denotes the sheaf of flat sections of
$W$ for the partial connection $D'$.

A transversely holomorphic section of a transversely holomorphic vector bundle $(V,\, 
D,\, \overline{\partial}_V)$ is a transversely holomorphic homomorphism to it from the 
trivial transversely holomorphic line bundle $(M\times{\mathbb C},\, d)$, where $d$ is 
the trivial connection on the trivial line bundle given by the de Rham differential (as 
shown above, a flat vector bundle gives a transversely holomorphic vector bundle).

\subsection{Transversely holomorphic projective bundle}\label{se3.2}

Fix an integer $r\, \geq\, 2$.
Take a $C^\infty$ principal $\text{PGL}(r,{\mathbb C})$--bundle
\begin{equation}\label{dp0}
p\ :\ P\ \longrightarrow\ M
\end{equation}
on the manifold $M$ with foliation $\mathcal F$. Let
\begin{equation}\label{dp}
dp\ :\ T^{\mathbb R}P\ \longrightarrow\ p^*T^{\mathbb R}M
\end{equation}
be the differential of the projection $p$ in \eqref{dp0}. A {\it partial integrable
connection on} $P$ {\it in the direction of} $\mathcal F$ is an integrable subbundle
\begin{equation}\label{e33}
{\mathcal F}^P\ \, \subset\ \, T^{\mathbb R}P
\end{equation}
satisfying the following two conditions:
\begin{enumerate}
\item The action of any $A\, \in\, \text{PGL}(r,{\mathbb C})$ on $P$ sends ${\mathcal 
F}^P$ to itself, and

\item the restriction of the differential $dp$ (see \eqref{dp}) to ${\mathcal F}^P\, 
\subset\, T^{\mathbb R}P$ is an isomorphism of ${\mathcal F}^P$ with $p^*{\mathcal F}\, 
\subset\, p^*T^{\mathbb R}M$.
\end{enumerate}

\begin{definition}\label{def0}
A {\it transversely principal} $\text{PGL}(r,{\mathbb C})$--bundle on $M$ is a $C^\infty$ principal
$\text{PGL}(r, {\mathbb C})$--bundle on $M$ equipped with a partial integrable connection in
the direction of the foliation $\mathcal F$.
\end{definition}

Let $(P,\, {\mathcal F}^P)$ be a transversely principal $\text{PGL}(r,{\mathbb C})$--bundle on $M$. Let
$$dp\otimes {\mathbb C}\ :\ T^{\mathbb C}P\ :=\ T^{\mathbb R}P\otimes{\mathbb C} 
\ \longrightarrow\ (p^*T^{\mathbb R}M)\otimes {\mathbb C} \ :=\ p^*T^{\mathbb C}M$$
be the complexification of the differential $dp$ in \eqref{dp}.
A transversely holomorphic structure on $(P,\, {\mathcal F}^P)$ is a $C^\infty$
integrable distribution
\begin{equation}\label{hs}
{\mathbb F}^P\ \, \subset\ \, T^{\mathbb C}P
\end{equation}
satisfying the following three conditions:
\begin{enumerate}
\item The action of any $A\, \in\, \text{PGL}(r,{\mathbb C})$ on $P$ sends ${\mathbb F}^P$ to itself,

\item ${\mathcal F}^P\ \, \subset\ \, {\mathbb F}^P$, and

\item the restriction of $dp\otimes {\mathbb C}$ to ${\mathbb F}^P$ is an isomorphism of
${\mathbb F}^P$ with $p^*\mathbb F$ (see \eqref{e8c}).
\end{enumerate}

Consider the action of $\text{PGL}(r,{\mathbb C})$ on ${\mathbb C}{\mathbb P}^{r-1}$ 
given by the standard action of $\text{SL}(r,{\mathbb C})$ on ${\mathbb C}^r$. In fact, 
the homomorphism \begin{equation}\label{eta} \eta\ :\ \text{PGL}(r,{\mathbb C})\ 
\longrightarrow\ \text{Aut}({\mathbb C}{\mathbb P}^{r-1}), \end{equation} where 
$\text{Aut}({\mathbb C}{\mathbb P}^{r-1})$ is the group of all holomorphic automorphisms 
of ${\mathbb C}{\mathbb P}^{r-1}$, given by this action of $\text{PGL}(r,{\mathbb C})$ 
on ${\mathbb C}{\mathbb P}^{r-1}$, is actually an isomorphism. Given a principal 
$\text{PGL}(r,{\mathbb C})$--bundle $p\, :\, P\, \longrightarrow\, M$ as in
\eqref{dp0}, we have the associated fiber bundle
\begin{equation}\label{eta2}
p^\eta\ :\ P^\eta\ :=\ P\times^\eta {\mathbb C}{\mathbb P}^{r-1}\ \longrightarrow\ M
\end{equation}
whose typical fiber is ${\mathbb C}{\mathbb P}^{r-1}$. We recall that $P^\eta$
is the quotient of $P\times {\mathbb C}{\mathbb P}^{r-1}$ where two points $(z_1,\, v_1),\, (z_2,\, v_2)\, \in\, 
P\times {\mathbb C}{\mathbb P}^{r-1}$ are identified if there is an element $g\, \in\, \text{PGL}(r,{\mathbb C})$
such that $z_2\,=\, z_1g$ and $v_2\,=\, \eta(g^{-1})(v_1)$, where $\eta$ is the homomorphism in
\eqref{eta}. Note that each fiber of $P^\eta$ is complex manifold which
is biholomorphic to ${\mathbb C}{\mathbb P}^{r-1}$.

Let $(P,\, {\mathcal F}^P)$ be a transversely principal $\text{PGL}(r,{\mathbb 
C})$--bundle on $M$. Then ${\mathcal F}^P$ induces an integrable partial connection, in 
the direction of $\mathcal F$, on any fiber bundle associated to $P$. In particular, 
${\mathcal F}^P$ induces an integrable partial connection on the associated fiber bundle 
$P^\eta$ in \eqref{eta2}; let $\widetilde{\mathcal F}^P\, \subset\, T^{\mathbb R} 
P^\eta$ be the integrable distribution on $P^\eta$ defining this partial connection 
given by ${\mathcal F}^P$. This partial connection $\widetilde{\mathcal F}^P$ on 
$P^\eta$ has the following property. Take any $C^\infty$ section $v$ of $\mathcal F$ 
defined on an open subset $U\, \subset\, M$; let $\widetilde v\, \in\, 
C^\infty(P^\eta\big\vert_U,\, \widetilde{\mathcal F}^P)$ be the unique horizontal lift 
of $v$ on $(p^\eta)^{-1}(U)\,=\, P^\eta\big\vert_U$, where $p^\eta$ is
the projection in \eqref{eta2}. Take any $C^\infty$ vertical vector field $w$ on
$P^\eta\big\vert_U$ for the projection $p^\eta$. Then the following holds:
\begin{equation}\label{crc}
[{\widetilde v},\, J^P(w)]\ =\ J^P([{\widetilde v},\, w]),
\end{equation}
where $J^P$ is the almost complex structure on the fibers of the ${\mathbb C}{\mathbb P}^1$--bundle $P^\eta$
(recall that every fiber of $P^\eta$ is a complex manifold biholomorphic to ${\mathbb C}{\mathbb P}^{r-1}$);
note that $[{\widetilde v},\, w]$ is a vertical vector field on $P^\eta$, because
$$dp^\eta([{\widetilde v},\, w])\ =\ [dp^\eta({\widetilde v}),\, dp^\eta(w)]\ =\ 0 ,$$
where $p^\eta$ is the projection in \eqref{eta2}, because $dp^\eta(w)\,=\, 0$ (recall
that $w$ is a vertical vector field on $P^\eta\big\vert_U$).

It is straightforward to check that an integrable partial connection, in the direction of $\mathcal F$,
on the associated projective bundle $P^\eta$ arises from some integrable subbundle
${\mathcal F}^P\, \subset\, T^{\mathbb R}P$ --- such that $(P,\, {\mathcal F}^P)$ is a 
transversely principal $\text{PGL}(r,{\mathbb C})$--bundle --- if and only if the integrable partial
connection on $P^\eta$ satisfies the condition in \eqref{crc}.

A transversely holomorphic structure ${\mathbb F}^P$ on a transversely principal $\text{PGL}(r,{\mathbb
C})$--bundle $(P,\, {\mathcal F}^P)$ (as in \eqref{hs}) produces an integrable partial
connection on
$P^\eta$ (see \eqref{eta2}) in the direction of $\mathbb F$ (constructed in \eqref{e8c});
let $\widetilde{\mathbb F}^P\, \subset\, T^{\mathbb C} P^\eta$ be the integrable distribution
on $P^\eta$ defining this partial connection in the direction of
${\mathbb F}$ given by ${\mathbb F}^P$. This partial connection
$\widetilde{\mathbb F}^P$ on $P^\eta$ has the following property. Take any $C^\infty$ section $v$
of $\mathbb F$ defined on an open subset $U\, \subset\, M$; let $\widetilde v\, \in\,
C^\infty(P^\eta\big\vert_U,\, \widetilde{\mathbb F}^P)$ be the unique horizontal lift of $v$. Take any
$C^\infty$ vertical vector field $w$ on $P^\eta\big\vert_U$. Then the following holds:
\begin{equation}\label{crc2}
[{\widetilde v},\, J^P(w)]\ =\ J^P([{\widetilde v},\, w]),
\end{equation}
where $J^P$ is as in \eqref{crc}; note that the vector field $[{\widetilde v},\, w]$ is vertical.

An integrable partial connection, in the direction of $\mathbb F$,
on $P^\eta$ arises from some transversely holomorphic structure
${\mathbb F}^P\, \subset\, T^{\mathbb C}P$ on $(P,\, {\mathcal F}^P)$
if and only if the integrable partial connection on $P^\eta$
satisfies the condition in \eqref{crc2}.

\section{Projective structure and projective connection}\label{se4}

Let
\begin{equation}\label{p1}
p_1\ :\ {\mathcal P}_0\ :=\ {\mathbb C}{\mathbb P}^1\times{\mathbb C}{\mathbb P}^1\ \longrightarrow\
{\mathbb C}{\mathbb P}^1
\end{equation}
be the projection to the first factor. So ${\mathcal P}_0$ is the trivial ${\mathbb C}{\mathbb P}^1$--bundle
on ${\mathbb C}{\mathbb P}^1$. The trivial connection on this trivial bundle ${\mathcal P}_0$ will be
denoted by
\begin{equation}\label{p2}
\mathbf{\nabla}_0.
\end{equation}

Let ${\mathcal P}$ be a transversely projective structure on the transversely complex foliation
$(M,\, {\mathcal F},\, \widehat{\mathcal F})$. Fix data $\{(U_i,\, \sigma_i)\}_{i\in I}$ giving
$\mathcal P$. More precisely, the collection $\{(U_i,\, \sigma_i)\}_{i\in I}$ is as in \eqref{e5} such that
\begin{itemize}

\item all $U_i\, \subset\, M$, $i\,\in\, I$, are simply connected open subsets, and
$\bigcup_{i\in I} U_i\,=\, M$,

\item{} each $\sigma_i$ is a $C^\infty$ submersion of $U_i$ to an open subset $D_i$ of ${\mathbb C}
{\mathbb P}^1$ satisfying the condition that the subbundle ${\mathcal F}\big\vert_{U_i} \, \subset\,
T^{\mathbb R} U_i$ coincides with the kernel of the differential map $d\sigma_i\, :\, TU_i \,
\longrightarrow \, \sigma^*_i TD_i$ for the submersion $\sigma_i$,

\item for each $i\, \in\, I$, the pullback, using the differential $d\sigma_i$, of the almost complex
structure on $\sigma_i(U_i)\, \subset\, {\mathbb C}{\mathbb P}^1$ coincides with the restriction, to
$U_i$, of the almost complex structure on the normal bundle $N$ given by $\widehat{\mathcal F}$, and

\item the holomorphic maps $f^c_{i,j}$ in condition \eqref{e6} are
of the form $z \,\longmapsto\, (az+b)/(cz+d)$, where $a,\,b,\,c,\,d\, \in\, \mathbb C$
with $ad-bc \,=\,1$.
\end{itemize}

For each $i\, \in\, I$, consider the pulled back ${\mathbb C}{\mathbb P}^1$--bundle
$\sigma^*_i {\mathcal P}_0\, \longrightarrow\, U_i$ (see \eqref{p1}) equipped with the pulled back flat 
connection $\sigma^*_i\mathbf{\nabla}_0$ (see \eqref{p2}). For any ordered pair $i,\, j\, \in\, I$
with $U_i\bigcap U_j\,\not=\, \emptyset$, and any connected component
$U^c_{i,j}\, \subset\, U_i\cap U_j$, over the open subset $U^c_{i,j}$ identify 
the ${\mathbb C}{\mathbb P}^1$--bundle $\sigma^*_i {\mathcal P}_0$ with the ${\mathbb C}
{\mathbb P}^1$--bundle $\sigma^*_j {\mathcal P}_0$ using the map
\begin{equation}\label{e10}
f^c_{i,j}\times f^c_{i,j} \ :\ (\sigma^*_i {\mathcal P}_0)\big\vert_{U^c_{i,j}}\ =\
U^c_{i,j}\times {\mathbb C}{\mathbb P}^1 \ \longrightarrow\
U^c_{i,j}\times {\mathbb C}{\mathbb P}^1 \ =\ (\sigma^*_j {\mathcal P}_0)\big\vert_{U^c_{i,j}}
\end{equation}
(see \eqref{e6}); note that $f^c_{i,j}\, \in\, \text{PGL}(2,{\mathbb C})$ and hence it gives
an automorphism of ${\mathbb C}{\mathbb P}^1$. The isomorphism of ${\mathbb C}{\mathbb P}^1$--bundles
in \eqref{e10} takes the flat connection $\sigma^*_i\mathbf{\nabla}_0$ to the
flat connection $\sigma^*_j\mathbf{\nabla}_0$. For all ordered pair $i,\, j\, \,\in\, I$, let
$$
f_{i,j}\ :\ U_i\cap U_j\ \longrightarrow\ \text{PGL}(2, {\mathbb C})\ =\ \text{Aut}({\mathbb C}
{\mathbb P}^1)
$$
be the locally constant function that sends any connected component $U^c_{i,j}\, \subset\, U_i\cap U_j$
to $f^c_{i,j}$. Since this collection $\{f_{i,j}\}_{i,j\in I}$ satisfies
the cocycle condition, the locally defined flat ${\mathbb C}{\mathbb P}^1$--bundles
$\{(\sigma^*_i {\mathcal P}_0,\, \sigma^*_i\mathbf{\nabla}_0)\}_{i\in I}$ patch together compatibly ---
using \eqref{e10} --- to define a ${\mathbb C}{\mathbb P}^1$--bundle
\begin{equation}\label{e11}
\varphi\ :\ {\mathbb P}_{\mathcal P} \ \longrightarrow\ M
\end{equation}
equipped with a flat connection
\begin{equation}\label{e12}
\nabla^{\mathcal P}.
\end{equation}
Note that the pair $({\mathbb P}_{\mathcal P},\, \nabla^{\mathcal P})$ in \eqref{e11} and
\eqref{e12} gives a ${\mathbb C}{\mathbb P}^1$--bundle equipped with a transversely holomorphic
structure on the foliated manifold $(M,\, \widehat{\mathcal F})$.

Consider the holomorphic section of the ${\mathbb C}{\mathbb P}^1$--bundle in \eqref{p1}
\begin{equation}\label{e12b}
s_0\ :\ {\mathbb C}{\mathbb P}^1\ \longrightarrow\ {\mathcal P}_0,\ \ \,
z\ \longmapsto\ (z,\,z)
\end{equation}
whose image is the diagonal in ${\mathbb C}{\mathbb P}^1\times{\mathbb C}{\mathbb P}^1$.
For each $i\, \in\, I$, consider the pulled back section $$\sigma^*_i s_0 \ :\ U_i\
\longrightarrow\ \sigma^*_i {\mathcal P}_0\ =\ {\mathbb P}_{\mathcal P}\big\vert_{U_i}$$
(see \eqref{e12b} and \eqref{e11}). These locally defined sections $\sigma^*_i s_0$ of ${\mathbb P}_{\mathcal P}$
patch together compatibly to define a global section
\begin{equation}\label{e13}
s_{\mathcal P}\ :\ M \ \longrightarrow\ {\mathbb P}_{\mathcal P}
\end{equation}
of the projection $\varphi$ in \eqref{e11}. From the fact that the section $s_0$ in \eqref{e12b} is holomorphic
it follows immediately that the section $s_{\mathcal P}$ in \eqref{e13} is transversely holomorphic.

Let $T_{\varphi}\ \subset\ T^{\mathbb R}{\mathbb P}_{\mathcal P}$ be the relative tangent bundle
for the projection $\varphi$ in \eqref{e11}; in other words, $T_{\varphi}$ is the kernel of the
differential $d\varphi\, :\, T^{\mathbb R}{\mathbb P}_{\mathcal P}\, \longrightarrow\,
\varphi^* T^{\mathbb R}M$.

Let $ds_{\mathcal P}\,:\, TM\, \longrightarrow\, s^*_{\mathcal P} T^{\mathbb R}{\mathbb P}_{\mathcal P}$
be the differential of the map $s_{\mathcal P}$ in \eqref{e13}. Let
\begin{equation}\label{e14}
{\mathcal S}_{\mathcal P}\ :\ TM\ \longrightarrow\ s^*_{\mathcal P} T_{\varphi}
\end{equation}
be the map that sends any $v\,\in\, T_mM$ to $ds_{\mathcal P}(v) -\widetilde{v}$, where $\widetilde{v}\, \in\,
T^{\mathbb R}_{s_{\mathcal P}(m)} {\mathbb P}_{\mathcal P}$ is the horizontal lift of $v$ for the
connection $\nabla^{\mathcal P}$ in \eqref{e12}. It can be shown that
\begin{equation}\label{e15}
\text{kernel}({\mathcal S}_{\mathcal P})\,\ =\,\ {\mathcal F},
\end{equation}
where ${\mathcal S}_{\mathcal P}$ is the homomorphism in \eqref{e14}.
Indeed, from the constructions of $s_{\mathcal P}$ and $\nabla^{\mathcal P}$ it follows immediately that
${\mathcal F}\, \subset\, \text{kernel}({\mathcal S}_{\mathcal P})$. On the other hand, from
the fact that the diagonal in ${\mathbb C}{\mathbb P}^1\times
{\mathbb C}{\mathbb P}^1$ is transversal to every constant section of the trivial ${\mathbb C}
{\mathbb P}^1$--bundle ${\mathcal P}_0$ in \eqref{p1} it follows that
$\text{kernel}({\mathcal S}_{\mathcal P}) \, \subset\, {\mathcal F}$. Thus, \eqref{e15} holds.

The condition in \eqref{e15} will be expressed in words by saying that
the section $s_{\mathcal P}$ satisfies the \textit{transversality condition} for the
connection $\nabla^{\mathcal P}$.

The transversely projective structure $\mathcal P$ can be reconstructed back from the 
triple $({\mathbb P}_{\mathcal P},\, \nabla^{\mathcal P},\, s_{\mathcal P})$ constructed 
in \eqref{e11}, \eqref{e12} and \eqref{e13}. More precisely, it will be shown that any 
triple $({\mathbb P}_{\mathcal P},\, \nabla^{\mathcal P},\, s_{\mathcal P})$ satisfying 
the above conditions produces a unique transversely projective structure.

To prove this, let
\begin{equation}\label{v1}
\varpi\ :\ {\mathbb P}\ \longrightarrow\ M
\end{equation}
be a ${\mathbb C}{\mathbb P}^1$--bundle on $M$
equipped with a flat connection $\nabla$, and let
$$
s\ :\ M \ \longrightarrow\ {\mathbb P}
$$
be a transversely holomorphic section, for the transversely holomorphic
foliation $\widehat{\mathcal F}$, that satisfies the following condition:
Consider the homomorphism
$$
{\mathcal S}\ :\ TM\ \longrightarrow\ s^* T_{\varpi},
$$
where $T_{\varpi}\, \subset\, T^{\mathbb R}{\mathbb P}$ is the relative tangent bundle 
for the projection $\varpi$ in \eqref{v1}, that sends any $v\,\in\, T_mM$, $m\, \in\, 
M$, to $ds (v) -\widetilde{v}$, where $\widetilde{v}\, \in\, T^{\mathbb R}_{s(m)} 
{\mathbb P}$ is the horizontal lift of $v$ for the flat connection $\nabla$ on $\mathbb 
P$, and $ds\,:\, T^{\mathbb R}M\,\longrightarrow\, s^*T^{\mathbb R}{\mathbb P}$ is the 
differential of the map $s$. The condition on $s$ says that the kernel of the above 
homomorphism $\mathcal S$ coincides with ${\mathcal F}\, \subset\, T^{\mathbb R}M$; in 
other words, $s$ satisfies the transversality condition for the connection $\nabla$. We 
will now construct a transversely projective structure from this triple $({\mathbb P},
\, \nabla,\, s)$.

Take any connected simply connected open subset $U\, \subset\, M$. Denote by ${\mathbf P}^f$ the space
of all flat sections of ${\mathbb P}\big\vert_U\, \longrightarrow\, U$ for the flat connection
$\nabla\big\vert_U$. Note that ${\mathbf P}^f$ is a Riemann surface isomorphic to ${\mathbb C}{\mathbb P}^1$.
The evaluation map
$$
\psi\ :\ U\times {\mathbf P}^f \ \longrightarrow\ {\mathbb P}\big\vert_U,\ \ \, (u,\, \beta)\ \longmapsto
\ \beta(u)
$$
is an isomorphism of ${\mathbb C}{\mathbb P}^1$--bundles over $U$. Denoting
the natural projection $U\times {\mathbf P}^f\,\longrightarrow\, {\mathbf P}^f$
by $q_2$, consider the composition of maps
$$
\psi^U\ :=\ q_2\circ \psi^{-1}\ :\ {\mathbb P}\big\vert_U \ \longrightarrow\ {\mathbf P}^f.
$$
It is now straightforward to check that all pairs of the above form $(U,\, \psi^U)$ together define a
transversely projective structure on the transversely holomorphic foliation $(M,\, \widehat{\mathcal F})$.

The above observations are summarized in the following theorem (compare with \cite{LP} and with \cite[Lemma 4.3]{BD}):

\begin{theorem}\label{thm0}
Let $\widehat{\mathcal F}$ be a transversely complex structure on the foliation $\mathcal F$.
Giving a transversely projective on $M$ compatible with $\widehat{\mathcal F}$ is equivalent
to giving a triple of the form $((P,\, {\mathcal F}^P,\, {\mathbb F}^P),\, \nabla,\, s)$, where
\begin{itemize}
\item $(P,\, {\mathcal F}^P)$ is a $C^\infty$ transversely principal ${\rm PGL}(2,{\mathbb C})$--bundle
on $M$, and ${\mathbb F}^P$ is a transversely holomorphic structure on $(P,\,
{\mathcal F}^P)$ (see \eqref{hs}),

\item $\nabla$ is a flat connection on the principal ${\rm PGL}(2,{\mathbb C})$--bundle
$P$ such that the transversely holomorphic structure on $(P,\, {\mathcal F}^P)$ given by $\nabla$ coincides
with the one given by ${\mathbb F}^P$, and

\item $s\, :\, M\, \longrightarrow\, {\mathbb P}\, :=\, P\times^{{\rm PGL}(2,{\mathbb C})}
{\mathbb C}{\mathbb P}^1$ is a holomorphic section of the transversely holomorphic ${\mathbb C}{\mathbb
P}^1$--bundle on $M$ associated to $(P,\, {\mathcal F}^P,\, {\mathbb F}^P)$ that satisfies the
transversality condition for the connection on ${\mathbb P}$ induced by $\nabla$.
\end{itemize}
\end{theorem}

\section{Independence of transversely holomorphic projective bundle}

\subsection{Projective bundle of relative dimension one}

In Section \ref{se4} it was shown that giving a transversely projective structure is 
equivalent to giving a flat ${\mathbb C}{\mathbb P}^1$--bundle together with a 
transversely holomorphic section satisfying the transversality condition. In this 
section it will be shown that the isomorphism class of this transversely holomorphic 
${\mathbb C}{\mathbb P}^1$--bundle does not depend on the transversely projective 
structure.

Consider the first jet bundle $J^1(T{\mathbb C}{\mathbb P}^1)\, \longrightarrow\, {\mathbb C}
{\mathbb P}^1$ of the holomorphic tangent bundle of ${\mathbb C}{\mathbb P}^1$. We recall that
$J^1(T{\mathbb C}{\mathbb P}^1)$ fits in the following non-split short exact sequence of holomorphic
vector bundles on ${\mathbb C}{\mathbb P}^1$:
$$
0\, \longrightarrow\, {\mathcal O}_{{\mathbb C}{\mathbb P}^1}\, \longrightarrow\,
J^1(T{\mathbb C}{\mathbb P}^1)\, \longrightarrow\, T{\mathbb C}{\mathbb P}^1\,
\longrightarrow\,0.
$$
Moreover, this non-split short exact sequence determines $J^1(T{\mathbb C}{\mathbb P}^1)$ uniquely up to
a holomorphic isomorphism. The 
action of $\text{PGL}(2,{\mathbb C})$ on ${\mathbb C}{\mathbb P}^1$ (see \eqref{eta}) has a
natural lift to an action of $\text{PGL}(2,{\mathbb C})$ on $T{\mathbb C}{\mathbb P}^1$, which, in
turn, produces an action of $\text{PGL}(2,{\mathbb C})$ on the jet bundle $J^1(T{\mathbb C}{\mathbb P}^1)$.
Consequently, we get an action of $\text{PGL}(2,{\mathbb C})$ on the projective bundle
\begin{equation}\label{e16}
{\mathbb P}(J^1)\ :=\ {\mathbb P}(J^1(T{\mathbb C}{\mathbb P}^1)) \ \longrightarrow\
{\mathbb C}{\mathbb P}^1.
\end{equation}

Consider the trivial ${\mathbb C}{\mathbb P}^1$--bundle ${\mathcal P}_0$ over
${\mathbb C}{\mathbb P}^1$ in \eqref{p1}. The action $\eta$ of $\text{PGL}(2,{\mathbb C})$
on ${\mathbb C}{\mathbb P}^1$ in \eqref{eta} produces the diagonal action of $\text{PGL}(2,{\mathbb C})$
on ${\mathbb C}{\mathbb P}^1\times {\mathbb C}{\mathbb P}^1$. This action on
${\mathbb C}{\mathbb P}^1\times {\mathbb C}{\mathbb P}^1$ makes ${\mathcal P}_0$ a
$\text{PGL}(2,{\mathbb C})$--equivariant holomorphic ${\mathbb C}{\mathbb P}^1$--bundle
over ${\mathbb C}{\mathbb P}^1$.

\begin{lemma}\label{lem1}
There is a natural ${\rm PGL}(2,{\mathbb C})$--equivariant holomorphic isomorphism of the
holomorphic ${\rm PGL}(2,{\mathbb C})$--equivariant ${\mathbb C}{\mathbb P}^1$--bundle
${\mathcal P}_0$ in \eqref{p1} with the
holomorphic ${\rm PGL}(2,{\mathbb C})$--equivariant ${\mathbb C}{\mathbb P}^1$--bundle
${\mathbb P}(J^1)$ in \eqref{e16}.
\end{lemma}

\begin{proof}
The Lie algebra $ pgl(2, {\mathbb C})\,=\, sl(2, {\mathbb C})$ of $\text{PGL}(2,{\mathbb C})$
will be denoted by $\mathfrak g$. Using the action $\eta$ of $\text{PGL}(2,{\mathbb C})$
on ${\mathbb C}{\mathbb P}^1$ (see \eqref{eta}), we have
$$
{\mathfrak g}\ =\ H^0({\mathbb C}{\mathbb P}^1,\, T{\mathbb C}{\mathbb P}^1).
$$
Consider the trivial holomorphic vector bundle ${\mathbb C}{\mathbb P}^1\times{\mathfrak g}\,
\longrightarrow\, {\mathbb C}{\mathbb P}^1$ with fiber ${\mathfrak g}$. Let
\begin{equation}\label{e17}
I\ :\ {\mathbb C}{\mathbb P}^1\times{\mathfrak g}\ =\
{\mathbb C}{\mathbb P}^1\times H^0({\mathbb C}{\mathbb P}^1,\, T{\mathbb C}{\mathbb P}^1)\
\longrightarrow\ J^2(T{\mathbb C}{\mathbb P}^1)
\end{equation}
be the evaluation map to the second order jet bundle $J^2(T{\mathbb C}{\mathbb P}^1)$
that sends any $(x,\,v)\, \in\, {\mathbb C}{\mathbb P}^1\times H^0({\mathbb C}{\mathbb P}^1,
\, T{\mathbb C}{\mathbb P}^1)$ to the restriction of the section $v$ to the second order
infinitesimal neighborhood of the point $x$ of ${\mathbb C}{\mathbb P}^1$. It is
straightforward to check that $I$ is a holomorphic isomorphism of $\text{PGL}
(2,{\mathbb C})$--equivariant holomorphic vector bundles (the action of $\text{PGL}
(2,{\mathbb C})$ on ${\mathbb C}{\mathbb P}^1$ produces a $\text{PGL}
(2,{\mathbb C})$--equivariant structure on $J^2(T{\mathbb C}{\mathbb P}^1))$.

Consider the tautological line subbundle $\iota\, :\, L\, \hookrightarrow\, {\mathbb C}{\mathbb P}^1\times
{\mathbb C}^2$ whose fiber over any $x\, \in\, {\mathbb C}{\mathbb P}^1$ is the line in ${\mathbb C}^2$
represented by $x$. Let $$q\ :\ {\mathbb C}{\mathbb P}^1\times {\mathbb C}^2\ \longrightarrow\
Q\ :=\ ({\mathbb C}{\mathbb P}^1\times {\mathbb C}^2)/L$$ be the natural quotient map. So we have
the short exact sequence of holomorphic vector bundles on ${\mathbb C}{\mathbb P}^1$
\begin{equation}\label{f1}
0\, \longrightarrow\, L\, \stackrel{\iota}{\longrightarrow}\, {\mathbb C}{\mathbb P}^1\times
{\mathbb C}^2 \, \stackrel{q}{\longrightarrow}\, Q\, =\, ({\mathbb C}{\mathbb P}^1\times {\mathbb C}^2)/L
\, \longrightarrow\, 0.
\end{equation}
We have the homomorphism
\begin{equation}\label{e18}
\gamma\ :\ L\otimes Q^*\ =\ \text{Hom}(Q,\, L)\ \longrightarrow\ {\mathbb C}{\mathbb P}^1\times{\mathfrak g}
\end{equation}
that sends any $w_x\, \in\, \text{Hom}(Q,\, L)_x$, $x\, \in\, {\mathbb C}{\mathbb P}^1$, to
$\iota(x)\circ w_x\circ q(x)\, \in\, \text{End}({\mathbb C}^2)$, where $q$ and $\iota$
are the homomorphisms in \eqref{f1}. The image of the homomorphism
$\gamma$ in \eqref{e18} clearly lies in ${\mathfrak g}\, \subset\, \text{End}({\mathbb C}^2)$; in
fact, the image of $\gamma$ lies in the nilpotent locus of $\text{End}({\mathbb C}^2)$. It
can be shown that
\begin{equation}\label{e19}
({\mathbb C}{\mathbb P}^1\times{\mathfrak g} )/\gamma(\text{Hom}(Q,\, L))\ =\
\text{Hom}(L,\, {\mathbb C}{\mathbb P}^1\times {\mathbb C}^2)\ =\ 
({\mathbb C}{\mathbb P}^1\times {\mathbb C}^2)\otimes L^*.
\end{equation}
Indeed, the isomorphism in \eqref{e19} is given by the restriction map
$$
{\mathbb C}{\mathbb P}^1\times \text{End}({\mathbb C}^2)\ \longrightarrow\ 
\text{Hom}(L,\, {\mathbb C}{\mathbb P}^1\times {\mathbb C}^2)
$$
that sends any $(x,\, A)\, \in\, {\mathbb C}{\mathbb P}^1\times \text{End}({\mathbb C}^2)$
to the homomorphism $L_x\, \longrightarrow\, {\mathbb C}^2$ defined by $v\, \longmapsto\, A(v)$.

We have the following commutative diagram of $\text{PGL}(2,{\mathbb C})$--equivariant holomorphic
homomorphisms
\begin{equation}\label{e20}
\begin{matrix}
0 & \longrightarrow & \text{Hom}(Q,\, L) & \stackrel{\gamma}{\longrightarrow} &
{\mathbb C}{\mathbb P}^1\times{\mathfrak g} & \longrightarrow &
({\mathbb C}{\mathbb P}^1\times {\mathbb C}^2)\otimes L^* & \longrightarrow & 0\\
&& \,\,\, \Big\downarrow I'' && \,\,\, \Big\downarrow I && \,\,\, \Big\downarrow I'\\
0 & \longrightarrow & (T{\mathbb C}{\mathbb P}^1)^* & \longrightarrow &
J^2(T{\mathbb C}{\mathbb P}^1) & \longrightarrow & J^1(T{\mathbb C}{\mathbb P}^1) 
& \longrightarrow & 0
\end{matrix}
\end{equation}
where the top exact sequence is given by the isomorphism in \eqref{e19}, the bottom 
exact sequence is the canonical sequence of jet bundles and $I$ is the isomorphism in 
\eqref{e17}; the homomorphisms $I'$ and $I''$ in \eqref{e20} are induced by $I$. Both 
the homomorphisms $I'$ and $I''$ are isomorphisms because $I$ is so. We note that 
$\text{Hom}(Q,\, L)$ is the unique holomorphic line subbundle of ${\mathbb C}{\mathbb 
P}^1\times{\mathfrak g}$ preserved by the action of $\text{PGL}(2,{\mathbb C})$. Since 
the bottom exact sequence in \eqref{e20} is $\text{PGL}(2,{\mathbb C})$--equivariant, 
and $I$ is an isomorphism of $\text{PGL}(2,{\mathbb C})$--equivariant vector bundles, we 
conclude that $I$ takes the line subbundle $\text{Hom}(Q,\, L)\, \subset\, {\mathbb 
C}{\mathbb P}^1\times{\mathfrak g}$ to $(T{\mathbb C}{\mathbb P}^1)^* \, \subset\, 
J^2(T{\mathbb C}{\mathbb P}^1)$.

It is evident that the projective bundle ${\mathbb P}(({\mathbb C}{\mathbb P}^1\times {\mathbb C}^2)
\otimes L^*)$ for the vector bundle $({\mathbb C}{\mathbb P}^1\times {\mathbb C}^2)\otimes L^*$
in \eqref{e20} is holomorphically identified with ${\mathcal P}_0$ in \eqref{p1}. Moreover, this
isomorphism between ${\mathbb P}(({\mathbb C}{\mathbb P}^1\times {\mathbb C}^2)\otimes L^*)$ and
${\mathcal P}_0$ is $\text{PGL}(2,{\mathbb C})$--equivariant.

The isomorphism $I'$ in \eqref{e20} produces a $\text{PGL}(2,{\mathbb C})$--equivariant holomorphic
isomorphism of the holomorphic $\text{PGL}(2,{\mathbb C})$--equivariant ${\mathbb C}{\mathbb P}^1$--bundle
${\mathcal P}_0$ with the holomorphic $\text{PGL}(2,{\mathbb C})$--equivariant ${\mathbb C}
{\mathbb P}^1$--bundle ${\mathbb P}(J^1)$.
\end{proof}

Consider the vector bundle $({\mathbb C}{\mathbb P}^1\times {\mathbb C}^2)\otimes L^*$ in \eqref{e20}.
Tensoring the exact sequence in \eqref{f1} with $L^*$, the following short exact sequence of
holomorphic vector bundles on ${\mathbb C}{\mathbb P}^1$ is obtained:
\begin{equation}\label{e21}
0\, \longrightarrow\, {\mathcal O}_{{\mathbb C}{\mathbb P}^1}\, {\longrightarrow}\, ({\mathbb C}{\mathbb P}^1\times
{\mathbb C}^2)\otimes L^* \, {\longrightarrow}\, Q\otimes L^*\, \longrightarrow\, 0.
\end{equation}
The isomorphism $I'$ in \eqref{e20} produces the following commutative diagram of homomorphisms
\begin{equation}\label{e22}
\begin{matrix}
0 & \longrightarrow & {\mathcal O}_{{\mathbb C}{\mathbb P}^1} &\longrightarrow & ({\mathbb C}
{\mathbb P}^1\times{\mathbb C}^2)\otimes L^* & \longrightarrow & Q\otimes L^* & \longrightarrow & 0\\
&& \Big\Vert && \,\,\,\Big\downarrow I' && \,\,\, \Big\downarrow \widetilde{I}\\ 
0 & \longrightarrow & {\mathcal O}_{{\mathbb C}{\mathbb P}^1} &\longrightarrow & J^1(T{\mathbb C}{\mathbb P}^1)
&\longrightarrow & J^0(T{\mathbb C}{\mathbb P}^1)\,=\, T{\mathbb C}{\mathbb P}^1&\longrightarrow & 0
\end{matrix}
\end{equation}
where the top exact sequence is the one in \eqref{e21}, the bottom exact sequence is the 
canonical sequence of jet bundles and the homomorphism $\widetilde{I}$ is induced by 
$I'$; the homomorphism $\widetilde{I}$ is an isomorphism because $I'$ is so. Note that 
${\mathcal O}_{{\mathbb C}{\mathbb P}^1}$ is the unique holomorphic line subbundle of 
$({\mathbb C}{\mathbb P}^1\times{\mathbb C}^2)\otimes L^*$ preserved by the action of 
$\text{PGL}(2,{\mathbb C})$. Since the bottom exact sequence in \eqref{e21} is 
$\text{PGL}(2,{\mathbb C})$--equivariant, and $I'$ is an isomorphism of 
$\text{PGL}(2,{\mathbb C})$--equivariant vector bundles, we conclude that $I'$ takes the 
line subbundle ${\mathcal O}_{{\mathbb C}{\mathbb P}^1}\,\subset\, ({\mathbb C}{\mathbb 
P}^1\times{\mathbb C}^2)\otimes L^*$ in \eqref{e21} to ${\mathcal O}_{{\mathbb 
C}{\mathbb P}^1} \, \subset\, J^1(T{\mathbb C}{\mathbb P}^1)$.

The isomorphism $\widetilde{I}$ in \eqref{e22} coincides with the natural isomorphism 
between $T{\mathbb C}{\mathbb P}^1$ and $Q\otimes L^*$.

\begin{lemma}\label{lem2}
The isomorphism between the ${\mathbb C}{\mathbb P}^1$--bundles ${\mathcal P}_0$ and ${\mathbb P}(J^1)\,=\,
{\mathbb P}(J^1(T{\mathbb C}{\mathbb P}^1))$
in Lemma \ref{lem1} takes the section $s_0$ of ${\mathcal P}_0$ in \eqref{e12b} to the section of
${\mathbb P}(J^1)$ given by the line subbundle ${\mathcal O}_{{\mathbb C}{\mathbb P}^1}
\, \subset\, J^1(T{\mathbb C}{\mathbb P}^1)$ in \eqref{e22}.
\end{lemma}

\begin{proof}
Consider the ${\mathbb C}{\mathbb P}^1$--bundle
${\mathcal P}_0\,=\, {\mathbb C}{\mathbb P}^1\times
{\mathbb C}{\mathbb P}^1 \, \stackrel{p_1}{\longrightarrow}\, {\mathbb C}{\mathbb P}^1$ in
\eqref{p1}. It is canonically identified with the projective bundle
${\mathbb P}({\mathbb C}{\mathbb P}^1\times {\mathbb C}^2)$. This identification
between ${\mathcal P}_0$ and ${\mathbb P}({\mathbb C}{\mathbb P}^1\times {\mathbb C}^2)$
takes the section $s_0$ of ${\mathcal P}_0$ (see \eqref{e12b}) to the section of
${\mathbb P}({\mathbb C}{\mathbb P}^1\times {\mathbb C}^2)$ given by the line subbundle
$L\, \subset\, {\mathbb C}{\mathbb P}^1\times{\mathbb C}^2 $ in \eqref{f1}.

Recall that the short exact sequence in \eqref{e21} is obtained by tensoring the
exact sequence in \eqref{f1} by $L^*$. The projective bundle
${\mathbb P}(({\mathbb C}{\mathbb P}^1\times {\mathbb C}^2)\otimes L^*)$ is identified
with ${\mathcal P}_0\,=\, {\mathbb P}({\mathbb C}{\mathbb P}^1\times {\mathbb C}^2)$, and this identification
takes the section $s_0$ of ${\mathcal P}_0$ in \eqref{e12b} (given by the diagonal) to
the section of ${\mathbb P}(({\mathbb C}{\mathbb P}^1\times {\mathbb C}^2)\otimes L^*)$
given by the line subbundle ${\mathcal O}_{{\mathbb C}{\mathbb P}^1}\, \subset\,
({\mathbb C}{\mathbb P}^1\times {\mathbb C}^2)\otimes L^*$ in \eqref{e21}.
Now the lemma follows from the commutative diagram in \eqref{e22}.
\end{proof}

\subsection{Transversal jet bundle}

Let $(V,\, D)$ be a transversely complex vector bundle on the foliated manifold $(M,\, {\mathcal F})$.
As in Section \ref{se3.1}, denote by $\mathbb V$ the sheaf on $M$ given by the locally defined flat
sections of $V$ for the partial connection $D$. Let $\widehat{\mathcal F}$ be a transversely complex
structure on $\mathcal F$. Let
$$
\overline{\partial}_V \ :\ {\mathbb V} \ \longrightarrow\ {\mathbb V} \otimes_{{\mathcal C}({\mathcal 
F})} ({\mathbb N}^{0,1})^*
$$
be a Dolbeault operator giving a transversely holomorphic structure on $(V,\, D)$ (see \eqref{e8d}).

Let
\begin{equation}\label{e23}
\widehat{\mathbb V}\, \ :=\,\ \text{kernel}(\overline{\partial}_V)\ \,\subset \ \, \mathbb V
\end{equation}
be the subsheaf defined by the locally defined transversely holomorphic sections.

Consider the very special case where $(V,\, D,\, \overline{\partial}_V)$
are given by the trivial complex line bundle $M\times {\mathbb C}\,
\longrightarrow\, M$ equipped with the trivial flat connection given by the de Rham differential.
In this spacial case, the corresponding sheaf $\widehat{\mathbb V}$ in \eqref{e23}
will be denoted by ${\mathcal O}_{\widehat{\mathcal F}}$.

Getting back to the general case of $(V,\, D,\, \overline{\partial}_V)$, consider the sheaf of homomorphisms
${\rm Hom}(\widehat{\mathbb V},\, {\mathcal O}_{\widehat{\mathcal F}})$. For any nonnegative integer $k$, let
\begin{equation}\label{e24}
{\rm Diff}^k(\widehat{\mathbb V},\, {\mathcal O}_{\widehat{\mathcal F}})\, \ \subset\,\
{\rm Hom}(\widehat{\mathbb V},\, {\mathcal O}_{\widehat{\mathcal F}})
\end{equation}
be the subsheaf defined by the homomorphisms that are locally differential operators of order less
than or equal to $k$. It is straightforward to check that
${\rm Diff}^k(\widehat{\mathbb V},\, {\mathcal O}_{\widehat{\mathcal F}})$ defined in \eqref{e24} is a
transversely holomorphic vector bundle on $(M,\, \widehat{\mathcal F})$ of rank $(k+1)\cdot\text{rank}(V)$.
The dual transversely holomorphic vector bundle
\begin{equation}\label{e25}
J^k_{\widehat{\mathcal F}}(V)\, \ :=\, \ ({\rm Diff}^k(\widehat{\mathbb V},\, {\mathcal O}_{\widehat{\mathcal F}}))^*
\end{equation}
will be called the \textit{transversal jet bundle of order} $k$ of the transversely holomorphic 
vector bundle $(V,\, D,\, \overline{\partial}_V)$.

Consider the transversely holomorphic line bundle $N^{1,0}$ (see \eqref{ej2}). Recall that
it has the partial connection ${}^N\nabla^{1,0}$ (see \eqref{e7}), and its transversely
holomorphic structure is given by the combination of the Lie bracket of vector fields on $M$
and the projection of vector fields to $N^{1,0}$.

We note that the transversely holomorphic vector bundle $J^0_{\widehat{\mathcal F}}(V)$ in \eqref{e25}
is the transversely holomorphic vector bundle $(V,\, D,\, \overline{\partial}_V)$ itself, and for any
integer $k\, \geq\, 1$, the transversely holomorphic vector bundle $J^k_{\widehat{\mathcal F}}(V)$
fits in the following short exact sequence of transversely holomorphic vector bundles
\begin{equation}\label{e26}
0 \ \longrightarrow\ V\otimes ((N^{1,0})^*)^{\otimes k} \ \longrightarrow\ 
J^k_{\widehat{\mathcal F}}(V)\ \longrightarrow\ J^{k-1}_{\widehat{\mathcal F}}(V)
\ \longrightarrow\ 0.
\end{equation}
The projection $J^k_{\widehat{\mathcal F}}(V)\ \longrightarrow\ J^{k-1}_{\widehat{\mathcal F}}(V)$
in \eqref{e26} is the dual of the natural inclusion map
$$
{\rm Diff}^{k-1}(\widehat{\mathbb V},\, {\mathcal O}_{\widehat{\mathcal F}})\ \hookrightarrow
\ {\rm Diff}^k(\widehat{\mathbb V},\, {\mathcal O}_{\widehat{\mathcal F}});
$$
the inclusion map $V\otimes ((N^{1,0})^*)^{\otimes k} \, \hookrightarrow\, 
J^k_{\widehat{\mathcal F}}(V)$ in \eqref{e26} is the dual of the symbol map
$$
{\rm Diff}^k(\widehat{\mathbb V},\, {\mathcal O}_{\widehat{\mathcal F}})\
\ \longrightarrow\ V^*\otimes (N^{1,0})^{\otimes k}
$$
defined on the differential operators.

\subsection{Projective bundle for transversely projective structures}

Consider the short exact sequence in \eqref{e26} for $k\,=\,1$ and $V\,=\, N^{1,0}$:
\begin{equation}\label{e27}
0 \ \longrightarrow\ {\mathcal O}_{\widehat{\mathcal F}} \ \longrightarrow\ 
J^1_{\widehat{\mathcal F}}(N^{1,0})\ \longrightarrow\ N^{1,0}
\ \longrightarrow\ 0.
\end{equation}
Let
\begin{equation}\label{e28}
{\mathbb P}_{\widehat{\mathcal F}}\ :=\ {\mathbb P}(J^1_{\widehat{\mathcal F}}(N^{1,0}))
\ \longrightarrow\ M
\end{equation}
be the corresponding transversely holomorphic ${\mathbb C}{\mathbb P}^1$--bundle. The line
subbundle ${\mathcal O}_{\widehat{\mathcal F}} \, \hookrightarrow\, 
J^1_{\widehat{\mathcal F}}(N^{1,0})$ in \eqref{e27} produces a transversely holomorphic section
\begin{equation}\label{e29}
{\mathcal S}_{\widehat{\mathcal F}}\ :\ M \ \longrightarrow\ {\mathbb P}_{\widehat{\mathcal F}}
\end{equation}
of the projective bundle ${\mathbb P}_{\widehat{\mathcal F}}$ in \eqref{e28}.

\begin{theorem}\label{thm1}
Let ${\mathcal P}$ be transversely projective structure on the transversely complex foliation
$\widehat{\mathcal F}$. Let ${\mathbb P}_{\mathcal P} \, \longrightarrow\, M$ be the transversely
${\mathbb C}{\mathbb P}^1$--bundle (see \eqref{e11}) corresponding to $\mathcal P$, and let
$$
s_{\mathcal P}\ :\ M \ \longrightarrow\ {\mathbb P}_{\mathcal P}
$$
be the section associated to $\mathcal P$ (see \eqref{e13}). Then the transversely
${\mathbb C}{\mathbb P}^1$--bundle ${\mathbb P}_{\mathcal P}$ is canonically holomorphically identified with
${\mathbb P}_{\widehat{\mathcal F}}$ in \eqref{e28}. This identification takes the section
$s_{\mathcal P}$ of ${\mathbb P}_{\mathcal P}$ to the section
${\mathcal S}_{\widehat{\mathcal F}}$ of ${\mathbb P}_{\widehat{\mathcal F}}$ in \eqref{e29}.
\end{theorem}

\begin{proof}
Recall the construction of the triple $({\mathbb P}_{\mathcal P},\, \nabla^{\mathcal P},\,
s_{\mathcal P})$ from ${\mathcal P}$ (done in \eqref{e11}, \eqref{e12} and \eqref{e13}).
Take an open subset $U\, \subset\, M$ and a submersion $\sigma\, :\, U\, \longrightarrow\,
{\mathbb C}{\mathbb P}^1$ which is compatible with the transversely projective structure $\mathcal P$.
Then the transversely holomorphic line bundle $N^{1,0}$ (see \eqref{ej2}) is identified with
$\sigma^* T^{1,0}{\mathbb C}{\mathbb P}^1$. We have
${\mathbb P}_{\mathcal P}\big\vert_U\,=\, \sigma^*{\mathcal P}_0$, where ${\mathcal P}_0$ is
the projective bundle in \eqref{p1}. The integrable connection $\nabla^{\mathcal P}\big\vert_U$ is
the pullback $\sigma^* \mathbf{\nabla}_0$, where $\mathbf{\nabla}_0$ is the integrable connection
on ${\mathcal P}_0$ constructed in \eqref{p2}. The section $s_{\mathcal P}\big\vert_U$ is the pulled
back section $\sigma^* s_0$, where $s_0$ is the section of ${\mathcal P}_0$ in \eqref{e12b}.

Thus, from Lemma \ref{lem1} it follows immediately that the transversely ${\mathbb 
C}{\mathbb P}^1$--bundle ${\mathbb P}_{\mathcal P}$ is holomorphically identified with 
the transversely ${\mathbb C}{\mathbb P}^1$--bundle ${\mathbb P}_{\widehat{\mathcal F}}$ 
in \eqref{e28}. The key point Lemma \ref{lem1} is that
the isomorphism between ${\mathcal P}_0$ and ${\mathbb P}(J^1)$ is
${\rm PGL}(2,{\mathbb C})$--equivariant.

Similarly, from Lemma \ref{lem2} it follows immediately that the above identification
between ${\mathbb P}_{\mathcal P}$ and ${\mathbb P}_{\widehat{\mathcal F}}$ takes the
section $s_{\mathcal P}$ of ${\mathbb P}_{\mathcal P}$ to the section
${\mathcal S}_{\widehat{\mathcal F}}$ of ${\mathbb P}_{\widehat{\mathcal F}}$ in \eqref{e29}.
\end{proof}

\section{Space of transversely projective structures}

As before, $(M,\, F)$ is a foliated smooth manifold of real codimension two, and 
$\widehat{\mathcal F}$ is a transversely complex structure on the foliation. For 
notational convenience, the transversely holomorphic line bundle $(N^{1,0})^*$ (see 
\eqref{ej2}) will be denoted by $\mathbb K$. Let
\begin{equation}\label{e30}
H^0_{\mathcal F}(M,\, {\mathbb K}^{\otimes 2})
\end{equation}
be the space of all transversely holomorphic sections of ${\mathbb K}^{\otimes 2}$.

\begin{remark}
It is known that not every foliation $(M,\, \widehat{\mathcal F})$ admits a transversely 
complex projective structure compatible with $\widehat{\mathcal F}$; see \cite{BD}, 
\cite{FMP}.
\end{remark} 

The following proposition generalizes \cite[Lemma 4.1]{BD}:

\begin{proposition}\label{prop1}
Assume that the transversely holomorphic foliation $(M,\, \widehat{\mathcal F})$ admits a
transversely complex projective structure compatible with $\widehat{\mathcal F}$. Then the space of
all transversely projective structures compatible with $\widehat{\mathcal F}$ is an affine space
for the complex vector space $H^0_{\mathcal F}(M,\, {\mathbb K}^{\otimes 2})$ in \eqref{e30}.
\end{proposition}

\begin{proof}
Let $X$ be a Riemann surface. Its holomorphic cotangent bundle will be denoted by $K_X$. It is known that the 
space of all complex projective structures on $X$ (compatible with the complex structure of $X$) is
an affine space for the vector space $H^0(X,\, K^{\otimes 2}_X)$ \cite{Gu1, Gu2} (see also \cite{LM} and \cite[Chapter 4]{StG}, revisiting and building upon the pioneering work of Schwarz on the subject). To reprove this fact here, consider the
principal $\text{PGL}(2,{\mathbb C})$--bundle $P_{\rm PGL}$ associated to the projective
bundle ${\mathbb P}(J^1(TX))$. We will show that
\begin{equation}\label{e31}
K_X\ \hookrightarrow\ \text{ad}(P_{\rm PGL}),
\end{equation}
where $\text{ad}(P_{\rm PGL})$ is the adjoint bundle of $P_{\rm PGL}$. For this, consider the natural
short exact sequence of jet bundles
\begin{equation}\label{js}
0\ \longrightarrow\ TX\otimes K_X\ =\ {\mathcal O}_X \ \stackrel{\iota}{\longrightarrow}\ J^1(TX)\
\stackrel{q}{\longrightarrow}\ J^0(TX) \ = TX \ \longrightarrow\ 0.
\end{equation}
We have the fiberwise injective homomorphism
$$
\Psi\ :\ K_X\ =\ \text{Hom}(TX,\, {\mathcal O}_X)\ \longrightarrow\ \text{End}(J^1(TX))
$$
defined by $\alpha\, \longmapsto\, \iota\circ\alpha\circ q$, where $\iota$ and $q$
are the maps in \eqref{js}. This homomorphism $\Psi$ gives the inclusion map
in \eqref{e31}.

Tensoring both sides of \eqref{e31} and taking global sections, we have
\begin{equation}\label{e32}
H^0(X,\, K^{\otimes 2}_X)\ \subset\ H^0(X,\, \text{ad}(P_{\rm PGL})\otimes K_X).
\end{equation}
Since the space of all holomorphic connections on $P_{\rm PGL}$ is an
affine space for $H^0(X,\, \text{ad}(P_{\rm PGL})\otimes K_X)$, the
subspace $H^0(X,\, K^{\otimes 2}_X)$ in \eqref{e32} acts
on the space of all holomorphic connections on $P_{\rm PGL}$.

There is a natural injective map from the space of all projective structures on $X$ 
(compatible with the complex structure of $X$) to the space of all holomorphic 
connections on $P_{\rm PGL}$; this follows from Lemma
\ref{lem1}. The image of this injective map is
an affine space for the vector space $H^0(X,\, K^{\otimes 2}_X)$ (it was shown
above that $H^0(X,\, K^{\otimes 2}_X)$ acts
on the space of all holomorphic connections on $P_{\rm PGL}$).
Using these two facts and \eqref{e32} it is deduced that 
the space of all projective structures on $X$ (compatible with the complex structure of 
$X$) is an affine space for the vector space $H^0(X,\, K^{\otimes 2}_X)$.

Take a transversely projective structure $P_1$ on the transversely holomorphic foliation
$(M,\, \widehat{\mathcal F})$. Fix data $\{(U_i,\, \sigma_i)\}_{i\in I}$ that gives $P_1$ (see
Definition \ref{def2}). Take another transversely projective structure $P_2$ on the
transversely holomorphic foliation $(M,\, \widehat{\mathcal F})$
given by $\{(U_i,\, \widehat{\sigma}_i)\}_{i\in I}$ (we may refine the
open covers of $M$ to assume that the two open covers coincide). From what is shown
above for projective structures on $X$ it follows that
$$P_1\big\vert_{U_i}- P_2\big\vert_{U_i}\ \,=\,\ \sigma^*_i \theta_i
$$
for every $i\, \in\, I$, where $\theta_i\, \in\, H^0(\sigma_i(U_i),\, K^{\otimes 2}_{\sigma_i(U_i)})$.
These locally defined sections $\theta_i$ patch together compatibly to give an element of
$H^0_{\mathcal F}(M,\, {\mathbb K}^{\otimes 2})$.

Conversely, given $\{(U_i,\, \sigma_i)\}_{i\in I}$ defining $P_1$, and $\theta'\, \in\,
H^0_{\mathcal F}(M,\, {\mathbb K}^{\otimes 2})$, consider the transversely projective structure
$P_1\big\vert_{U_i}+\theta'\big\vert_{U_i}$ on $U_i$ for every $i\,\in\, I$. These locally defined
transversely projective structures patch together compatibly to produce a 
transversely projective structure on entire $M$ for the 
$(M,\, \widehat{\mathcal F})$.

This way, the space of all transversely projective structures compatible with $\widehat{\mathcal F}$
is an affine space for $H^0_{\mathcal F}(M,\, {\mathbb K}^{\otimes 2})$.
\end{proof}

\section{Transversal opers}

As before, $\widehat{\mathcal F}$ is a transversely holomorphic foliation on $M$. Fix an integer
$r\, \geq\, 2$. Let ${\mathbf P}\, \longrightarrow\, M$ be a transversely holomorphic ${\mathbb C}
{\mathbb P}^{r-1}$--bundle on $M$ (see Section \ref{se3.2}). The fiber of $\mathbf P$ over any point
$x\, \in\, M$ will be denoted by ${\mathbf P}_x$. We have two foliations on $\mathbf P$ 
\begin{equation}\label{e34}
{\mathcal F}^{\mathbf P}\ \, \subset\ \, T^{\mathbb R}{\mathbf P}
\end{equation}
associated to the structure of a transversely ${\mathbb C}{\mathbb P}^{r-1}$--bundle (see \eqref{e33}) and
\begin{equation}\label{e35}
{\mathbb F}^{\mathbf P}\ \, \subset\ \, T^{\mathbb C}{\mathbf P}
\end{equation}
associated to the transversely holomorphic structure of ${\mathbf P}$ (see \eqref{hs});
note that ${\mathcal F}^{\mathbf P}\, \subset\, {\mathbb F}^{\mathbf P}$.

Take any nonnegative integer $j\, \leq\, r-1$. A \textit{holomorphic projective subbundle} of $\mathbf P$
of relative dimension $j$ is a $C^\infty$ subbundle $${\mathbf P}_j\ \subset\ {\mathbf P}$$ satisfying the
following three conditions:
\begin{enumerate}
\item For any $x\, \in\, M$, the fiber of ${\mathbf P}_j$ over $x$ is a complex
linear subspace of ${\mathbf P}_x$ of complex dimension $j$;

\item For any $x\, \in\, {\mathbf P}_j$, the tangent subspace $({\mathcal F}^{\mathbf P})_x
\, \subset\, T^{\mathbb R}_x {\mathbf P}$ (see \eqref{e34}) is contained in the subspace
$T^{\mathbb R}_x{\mathbf P}_j\, \subset\, T^{\mathbb R}_x{\mathbf P}$;

\item For any $x\, \in\, {\mathbf P}_j$, the tangent subspace $({\mathbb F}^{\mathbf 
P})_x\, \subset\, T^{\mathbb C}_x {\mathbf P}$ (see \eqref{e35}) is contained in the 
subspace $T^{\mathbb C}_x{\mathbf P}_j\, \subset\, T^{\mathbb C}_x{\mathbf P}$.
\end{enumerate}

\begin{remark}\label{rem1}
Note that the third condition implies the second condition because ${\mathcal F}^{\mathbf P}\,\subset\,
{\mathbb F}^{\mathbf P}$.\end{remark}
 A $C^\infty$ subbundle $${\mathbf P}_j\ \subset\ {\mathbf P}$$ satisfying the
first two condition (of the above three conditions) is called a $C^\infty$ projective subbundle of $\mathbf P$
of relative dimension $j$.

Let 
\begin{equation}\label{e36}
\varphi\, \ :\, \ P_{\rm PGL}\,\ \longrightarrow\,\ M
\end{equation}
be a transversely principal $\text{PGL}(r,{\mathbb C})$--bundle on $M$ (see Definition \ref{def0})
equipped with a flat connection $\nabla$ such that the restriction
of $\nabla$ in the direction of $\mathcal F$ coincides with the partial connection on $P_{\rm PGL}$
defining its transversal structure (see Definition \ref{def0}).

The flat connection $\nabla$ produces a transversely holomorphic structure on $P_{\rm PGL}$. This 
transversely holomorphic structure on $P_{\rm PGL}$ can be described as follows: Take the horizontal lift
$$
\widetilde{\nabla}\ \, :\ \, \varphi^*T^{\mathbb C}M\, \ \longrightarrow\,\ T^{\mathbb C}P_{\rm PGL}
$$
for $\nabla$, and restrict it to the subbundle $\varphi^*{\mathbb F}\, \subset\, 
\varphi^*T^{\mathbb C}M$ (see \eqref{e8c}). The subbundle $\widetilde{\nabla}
(\varphi^*{\mathbb F})\, \subset\, T^{\mathbb C}P_{\rm PGL}$ defines a transversely
holomorphic structure on $P_{\rm PGL}$ (see \eqref{hs}).

Let
\begin{equation}\label{e37}
p\,\ :\,\ {\mathbf P}\,\ \longrightarrow\ \, M
\end{equation}
the transversely ${\mathbb C}{\mathbb P}^{r-1}$--bundle on $M$ associated to the principal $\text{PGL}
(r,{\mathbb C})$--bundle $P_{\rm PGL}$ in \eqref{e36} (see \eqref{eta2} for the associated bundle).

The flat connection $\nabla$ on $P_{\rm PGL}$ produces a flat connection on the transversely ${\mathbb C}
{\mathbb P}^{r-1}$--bundle $\mathbf P$ in \eqref{e37}; this flat connection on $\mathbf P$ will be denoted
by $D$. Let
\begin{equation}\label{e39}
{\mathbb D}\,\ :\,\ p^*T^{\mathbb R}M\,\ \longrightarrow\,\ T^{\mathbb R}{\mathbf P}
\end{equation}
be the homomorphism defining the connection $D$. So the distribution ${\mathbb D}(p^*T^{\mathbb R}M)
\, \subset\, T^{\mathbb R}{\mathbf P}$ is the horizontal subspace for $D$. Note that the distribution
${\mathbb D}(p^*T^{\mathbb R}M)$ is integrable because the connection $D$ is flat.

The transversely ${\mathbb C}{\mathbb P}^{r-1}$--bundle $\mathbf P$ has a transversely 
holomorphic structure. Indeed, the transversely holomorphic structure on $P_{\rm PGL}$ 
induces a transversely holomorphic structure on the associated bundle $\mathbf P$. On 
the other hand, the flat connection $D$ on $\mathbf P$ produces a transversely 
holomorphic structure on $\mathbf P$. It is evident that these two transversely 
holomorphic structures on $\mathbf P$ actually coincide.

For every integer $0\, \leq\, j\, \leq\, r-2$, let
\begin{equation}\label{e38}
{\mathbf P}_j\,\ \subset\, \ {\mathbf P}
\end{equation}
be a holomorphic projective subbundle of $\mathbf P$ of relative dimension $j$ satisfying the following two conditions:
\begin{enumerate}
\item ${\mathbf P}_j\, \subset\, {\mathbf P}_{j+1}$ for all $0\, \leq\, j\, \leq\, r-2$.

\item For any $0\, \leq\, j\, \leq\, r-2$ and any $y\, \in\, {\mathbf P}_j$,
\begin{equation}\label{e44}
{\mathbb D}(T^{\mathbb R}_{p(y)} M)\,\ \subset\,\ T^{\mathbb R}_y {\mathbf P}_{j+1},
\end{equation}
where $\mathbb D$ is the homomorphism in \eqref{e39}
and $p$ is the projection in \eqref{e37}.
\end{enumerate}

Take any $0\, \leq\, j\, \leq\, r-2$. The manifold ${\mathbf P}_j$ is equipped with the foliation
\begin{equation}\label{e40}
{\mathcal F}_j\,\ :=\, \ {\mathcal F}^{\mathbf P}\big\vert_{{\mathbf P}_j}
\end{equation}
(see \eqref{e34}) because \eqref{e44} holds.

For any $r-1 \, \geq\, k\, >\, j$, consider the normal bundle
\begin{equation}\label{e41}
{\mathbf N}^k_j \,\ :=\, \ \left(T^{\mathbb R}{\mathbf P}_k\big\vert_{{\mathbf P}_j}\right)
\Big{/} T^{\mathbb R}{\mathbf P}_j\,\ \longrightarrow\,\ {\mathbf P}_j
\end{equation}
over ${\mathbf P}_j$. We will show that ${\mathbf N}^k_j$ has a transversely holomorphic structure for the
foliation ${\mathcal F}_j$ in \eqref{e40}.

First note that $(T^{\mathbb R}{\mathbf P}_j)/{\mathcal F}_j$ is a transversely 
holomorphic vector bundle; this transversely holomorphic structure is given by the 
transversely holomorphic structure of ${\mathbf P}_j$. Similarly, the quotient 
$\left(\left(T^{\mathbb R}{\mathbf P}_k \right)\big\vert_{{\mathbf 
P}_j}\right)\Big{/}{\mathcal F}_j$ is a transversely holomorphic vector bundle. Next 
note that $(T^{\mathbb R}{\mathbf P}_j)/{\mathcal F}_j$ is a transversely holomorphic 
subbundle of $\left(\left(T^{\mathbb R}{\mathbf P}_k\right)\big\vert_{{\mathbf 
P}_j}\right)\Big{/}{\mathcal F}_j$, and consequently, the corresponding quotient
\begin{equation}\label{e42}
\widetilde{\mathbf N}^k_j\,\ :=\,\ 
\left(\left(\left(T^{\mathbb R}{\mathbf P}_k\right)\big\vert_{{\mathbf P}_j}
\right)\Big{/}{\mathcal F}_j\right)\Big{/}\left((T^{\mathbb R}{\mathbf P}_j)/{\mathcal F}_j\right)
\end{equation}
is a transversely holomorphic vector bundle. The transversely real vector bundle underlying the
transversely holomorphic vector bundle $\widetilde{\mathbf N}^k_j$ in \eqref{e42} is identified with the
transversely real vector bundle ${\mathbf N}^k_j$ in \eqref{e41}. Using this identification between
${\mathbf N}^k_j$ and $\widetilde{\mathbf N}^k_j$, the
transversely holomorphic structure of $\widetilde{\mathbf N}^k_j$ induces a transversely holomorphic
structure on ${\mathbf N}^k_j$.

The restriction of the map $p$ in \eqref{e37} to ${\mathbf P}_j$ will be denoted by $p_j$.
Now we will construct a holomorphic homomorphism of transversely holomorphic vector bundles
on ${\mathbf P}_j$
\begin{equation}\label{e43}
\Phi_j\, \ :\,\ p^*_j N^{1,0}\,\ \longrightarrow\,\ {\mathbf N}^{j+1}_j
\end{equation}
for all $0\, \leq\, j\, \leq\, r-2$, where $N^{1,0}$ is the transversely holomorphic vector bundle in
\eqref{ej2}. To construct $\Phi_j$, take any $y\, \in\, {\mathbf P}_j$ and denote $p_j(y)\, \in\,
M$ by $x$. Take any $v\, \in\, (p^*_j N^{1,0})_y\,=\, (N^{1,0})_x$. Take any $w\, \in\,
T^{\mathbb C}_xM\,=\,T^{\mathbb R}_xM\otimes{\mathbb C}$ such that the projection of $w$
to $N\otimes{\mathbb C}$ (see \eqref{e3} and \eqref{ej2}) coincides with $v$. Now consider
${\mathbb D}(w)\,\in\, T^{\mathbb C}_y {\mathbf P}_{j+1}$ (see \eqref{e44}). Also, take any
$\widetilde{w}\, \in\, T^{\mathbb C}_y {\mathbf P}_j$ that projects to $w$; in other words,
$dp_j(\widetilde{w})\,=\, w$, where
\begin{equation}\label{e45}
dp_j\,\ :\,\ T^{\mathbb C}{\mathbf P}_j \,\ \longrightarrow\,\ p^*_j T^{\mathbb C}M
\end{equation}
is the differential of the projection $p_j$. Let
\begin{equation}\label{e46}
v'\,\ \in\,\ (\widetilde{\mathbf N}^k_j)_y
\end{equation}
be the image of $\widetilde{w}-{\mathbb D}(w)$ in $(\widetilde{\mathbf N}^k_j)_y$. We 
will show that $v'$ does not depend on the choices of the lifts $w$ and $\widetilde{w}$.

Any two choices of $\widetilde{w}$ differ by an element in the kernel of the 
homomorphism $dp_j$ in \eqref{e45}. So, $v'$ does not depend on the choice of the lift 
$\widetilde{w}$. Any two choices of $w$ differ by an element of ${\mathcal 
F}_x\otimes{\mathbb C}\, \subset\, T^{\mathbb C}_xM$. Since
$$
{\mathbb D}({\mathcal F}_x)\,\ \subset\,\ T^{\mathbb R}_y {\mathbf P}_{j},
$$
where $\mathbb D$ is the homomorphism in \eqref{e39}, it follows that
$v'$ does not depend on the choice of $w$. The homomorphism $\Phi_j$ in \eqref{e43} sends any
$v\, \in\, (p^*_j N^{1,0})_y$, $y\, \in\, {\mathbf P}_j$, to $v'$ constructed in \eqref{e46}.
Using the two facts --- that ${\mathbf P}_i$ are transversely holomorphic subbundles of $\mathbf P$
and the transversely holomorphic structure of ${\mathbf P}_i$ is compatible with the
flat connection $D$ on $\mathbb P$ --- it is straightforward to deduce that $\Phi_j$ is
a transversely holomorphic homomorphism.

\begin{definition}\label{def3}
A ${\rm PGL}(r,{\mathbb C})$ \textit{transversely oper} is a triple
$(P_{\rm PGL},\, \nabla, \, \{{\mathbf P}_j\}^{r-2}_{0})$ as above (see \eqref{e36} and
\eqref{e38}) such that the homomorphism $\Phi_j$ in \eqref{e43} is an isomorphism for all
$0\, \leq\, j\, \leq\, r-2$.
\end{definition}

To give examples of ${\rm PGL}(r,{\mathbb C})$ transversely opers, let
\begin{equation}\label{p3}
((P,\, {\mathcal F}^P,\, {\mathbb F}^P),\, \nabla,\, s)
\end{equation}
be a transversely
projective structure on the transversely holomorphic foliation $\widehat{\mathcal F}$
(see Theorem \ref{thm0}). Let
\begin{equation}\label{p}
p\ :\ {\mathbb P}_1\ :=\ P\times^{{\rm PGL}(2,{\mathbb C})} {\mathbb C}{\mathbb P}^1
\ \longrightarrow\ M
\end{equation}
be the transversely holomorphic ${\mathbb C}{\mathbb P}^1$--bundle associated to $(P,\, {\mathcal F}^P,
\, {\mathbb F}^P)$. We recall from Theorem \ref{thm0} that $s$
in \eqref{p3} is holomorphic section of ${\mathbb P}_1$
satisfying the transversality condition for the flat connection $D$ on ${\mathbb P}_1$ induced by the
flat connection $\nabla$ on $P$. Let
$$
{\mathbb P}_{r-1} \ :=\ \text{Sym}^{r-1}_p({\mathbb P}_1)\ \longrightarrow\ M
$$
be the relative $(r-1)$--th symmetric product for the projection $p$ in \eqref{p}. So 
${\mathbb P}_{r-1}$ is a ${\mathbb C}{\mathbb P}^{r-1}$--bundle on $M$. The flat 
connection $D$ on ${\mathbb P}_1$ induces a flat connection on ${\mathbb P}_{r-1}$; this 
induced flat connection on ${\mathbb P}_{r-1}$ will be denoted by $\mathbb D$.
Consider the principal $\text{PGL}(r,{\mathbb C})$--bundle on $M$ obtained by
extending the structure group of the principal $\text{PGL}(2,{\mathbb C})$--bundle
$P$ using the homomorphism $\text{PGL}(2,{\mathbb C})\, \longrightarrow\,
\text{PGL}(r,{\mathbb C})$ given by the natural action of $\text{GL}(2,{\mathbb C})
\,=\, \text{Aut}({\mathbb C}^2)$ on $\text{Aut}({\rm Sym}^{r-1}({\mathbb C}^2))$.
The ${\mathbb C}{\mathbb P}^{r-1}$--bundle over $M$ associated to this
principal $\text{PGL}(r,{\mathbb C})$--bundle is identified with ${\mathbb P}_{r-1}$.
So the flat connection $D$ on ${\mathbb P}_1$ induces a flat connection on ${\mathbb P}_{r-1}$

Also, the transversely holomorphic structure on ${\mathbb P}_1$ induces a 
transversely holomorphic structure on ${\mathbb P}_{r-1}$. We note that this
transversely holomorphic structure on ${\mathbb P}_{r-1}$ coincides with the one given by the
flat connection $\mathbb D$ on ${\mathbb P}_{r-1}$ induced by $D$. Indeed, this follows immediately from the
fact that the transversely holomorphic structure on ${\mathbb P}_{1}$ coincides with the one given by the
flat connection $D$.

The ${\mathbb C}{\mathbb P}^{r-1}$--bundle ${\mathbb P}_{r-1}$ gives a principal ${\rm PGL}(r,
{\mathbb C})$--bundle ${\mathbf P}$ on $M$. The fiber of ${\mathbf P}$ over any $x\, \in\, M$
is the space of all biholomorphisms from ${\mathbb C}{\mathbb P}^{r-1}$ to the fiber
$({\mathbb P}_{r-1})_x$ of ${\mathbb P}_{r-1}$ over $x$. The above flat connection $\mathbb D$
on ${\mathbb P}_{r-1}$ produces a flat connection on the ${\rm PGL}(r,
{\mathbb C})$--bundle ${\mathbf P}$; this flat connection on ${\mathbf P}$ will be denoted by $\nabla$.

For any integer $1\, \leq\, i\, \leq\, r-1$, let
\begin{equation}\label{p4}
p_i\ :\ {\mathbb P}'_{i} \ :=\ \text{Sym}^{r-1}_p({\mathbb P}_1)\ \longrightarrow\ M
\end{equation}
be the relative $i$--th symmetric product for the projection $p$. Let
So ${\mathbb P}_{i}$ is a
${\mathbb C}{\mathbb P}^{i}$--bundle on $M$ which is equipped with a transversely holomorphic structure
and also a flat connection compatible with the transversely holomorphic structure.

We have an embedding
$$
\Psi_i\ :\ {\mathbb P}'_{i} \ \longrightarrow\ {\mathbb P}_{r-1}
$$
that sends any $z\, \in\, {\mathbb P}'_{i}$ to $z+(r-1-i)\cdot s(p_i(z))$, where $s$ and $p_i$ are
constructed in \eqref{p3} and \eqref{p4} respectively. Let
$$
\Psi_0\ :\ M \ \longrightarrow\ {\mathbb P}_{r-1}
$$
be the map defined by $x\, \longmapsto\, (r-1)\cdot s(x)$.
Note that $\Psi_{r-1}$ is the identity map
of ${\mathbb P}_{r-1}$. For any integer $0\, \leq\, i\, \leq\, r-1$,
the image $\Psi_i({\mathbb P}'_{i})\,
\subset\, {\mathbb P}_{r-1}$ will be denoted by ${\mathbb P}_{i}$. 

The data $({\mathbf P},\, \nabla,\, \{{\mathbb P}_j\}_{j=0}^{r-2})$ define a
transversely ${\rm PGL}(r,{\mathbb C})$--oper on the transversely
complex foliation $(M,\, \widehat{\mathcal F})$.

\section{Independence of the projective bundle and filtration}

In this section the following will be shown: Take two ${\rm PGL}(r,{\mathbb C})$ transversely
opers $(P^1_{\rm PGL},\, \nabla^1, \, \{{\mathbf P}^1_j\}^{r-2}_{0})$ and
$(P_{\rm PGL}^2,\, \nabla^2, \, \{{\mathbf P}^2_j\}^{r-2}_{0})$ on the transversely
complex foliation $(M,\, \widehat{\mathcal F})$; see Definition \ref{def3}.
Let ${\mathbf P}^1_{r-1}$ (respectively, ${\mathbf P}^2_{r-1}$) be the transversely
holomorphic ${\mathbb C}{\mathbb P}^{r-1}$--bundle on $(M,\, \widehat{\mathcal F})$
associated to the transversely holomorphic principal $\text{PGL}(r,{\mathbb C})$--bundle
$P^1_{\rm PGL}$ (respectively, $P^2_{\rm PGL}$). Then there is a natural holomorphic
isomorphism $\Phi\, :\, {\mathbf P}^1_{r-1}\, \longrightarrow\, {\mathbf P}^2_{r-1}$
such that $\Phi({\mathbf P}^1_j)\,=\, {\mathbf P}^2_j$ for all $0\,\leq\, j\, \leq\, r-1$.

We first recall a property of projective bundles that plays a very useful role here.

Take a complex manifold $X$ and a holomorphic ${\mathbb C}{\mathbb P}^d$--bundle
$$
\gamma\ :\ {\mathcal P}\ \longrightarrow\ X
$$
on $X$. Let $T^\gamma\, \subset\, T^{1,0}{\mathcal P}$
be the holomorphic relative tangent bundle for the projection $\gamma$. In other words,
$T^\gamma$ is the kernel of the differential $d\gamma\, :\, T^{1,0}{\mathcal P}\, \longrightarrow\,
\gamma^* T^{1,0}X$ of the map $\gamma$. Consider the holomorphic vector bundle
$\gamma^*\gamma_* T^\gamma$ on ${\mathcal P}$. It has a homomorphism to the relative jet bundle
$$
{\mathcal I} \ :\ \gamma^*\gamma_* T^\gamma\ \longrightarrow\ J^1_\gamma (T^\gamma)
$$
given by the restriction of sections to the first order neighborhoods of points. Note that the
fiber of $J^1_\gamma (T^\gamma)$ over any $z\, \in\, {\mathcal P}$ is the fiber
$J^1(T^\gamma\big\vert_{\gamma^{-1}(\gamma (z))})_z$ over $z\, \in\, \gamma^{-1}(\gamma (z))$
of the first jet bundle of the holomorphic vector bundle $T^\gamma\big\vert_{\gamma^{-1}(\gamma (z))}\,
\longrightarrow\, \gamma^{-1}(\gamma (z))$. So $J^1_\gamma (T^\gamma)$ fits in the following short
exact sequence of holomorphic vector bundles on $\mathcal P$:
\begin{equation}\label{d1}
0\ \longrightarrow\ T^\gamma\otimes (T^\gamma)^*\ \longrightarrow\ J^1_\gamma (T^\gamma)
\ \longrightarrow\ T^\gamma \ \longrightarrow\ 0.
\end{equation}
Let $\text{ad}(T^\gamma)\, \subset\, \text{End}(T^\gamma)\,=\, T^\gamma\otimes (T^\gamma)^*$ be the
subbundle of co-rank one given by the sheaf of endomorphisms of trace zero. From \eqref{d1} we have
the following short
exact sequence of holomorphic vector bundles on $\mathcal P$:
\begin{equation}\label{d4}
0\, \longrightarrow\, {\mathcal O}_{\mathcal P} \,=\,
(T^\gamma\otimes (T^\gamma)^*)/\text{ad}(T^\gamma)\, \longrightarrow\, J^1_\gamma (T^\gamma)/
\text{ad}(T^\gamma)\, =:\, {\mathcal J}^1_\gamma (T^\gamma)
\, \longrightarrow\, T^\gamma \, \longrightarrow\, 0.
\end{equation}
Then there is a natural holomorphic isomorphism
\begin{equation}\label{d3}
\beta\ :\ \gamma^*{\mathcal P} \ \stackrel{\sim}{\longrightarrow}\ P({\mathcal J}^1_\gamma (T^\gamma)),
\end{equation}
where $P({\mathcal J}^1_\gamma (T^\gamma)$ is the projective bundle on $\mathcal P$ that parametrizes
the lines in the fibers of the vector bundle ${\mathcal J}^1_\gamma (T^\gamma)$ in \eqref{d4} and
$\gamma^*{\mathcal P}$ is the pullback of the ${\mathbb C}{\mathbb P}^d$--bundle.

To see the isomorphism in \eqref{d3}, take a complex vector space $W$ of dimension $d+1$. Denote by
$P(W)$ the projective space parametrizing the lines in $W$. We have $$H^0(P(W),\, T^{1,0}W)\ =\
\text{ad}(W)$$ (trace zero endomorphisms of $W$). Take a point $x\,\in\, P(W)$, and consider the
evaluation map
\begin{equation}\label{d5}
I_x \ : \ \text{ad}(W)\ =\ H^0(P(W),\, T^{1,0}W)\ \longrightarrow\ {\mathcal J}^1 (T^{1,0}W)_x ,
\end{equation}
where ${\mathcal J}^1 (T^{1,0}W)_x$ is the quotient of ${J}^1 (T^{1,0}W)_x$ constructed as in
\eqref{d4} (set the base manifold $X$ to be a point in \eqref{d4}). Let $L\, \subset\, W$ be the line in $W$
represented by $x$. Consider the quotient map $\alpha\, :\, \text{ad}(W)\, \longrightarrow\,
W\otimes L^*$ that sends any $A\,\in\, \text{ad}(W)$ to the homomorphism $L\, \longrightarrow \, W$
that maps any $\ell\, \in\, L$ to $A(\ell)$. The map $I_x$ in \eqref{d5} factors through $\alpha$, and
the resulting map $W\otimes L^*\,\longrightarrow\, {\mathcal J}^1 (T^{1,0}W)_x$ is an isomorphism.
Thus $P(W)\,=\, P(W\otimes L)\,=\, P({\mathcal J}^1 (T^{1,0}W)_x)$. This pointwise construction gives
the isomorphism $\beta$ in \eqref{d3}.

Take a transversely ${\rm PGL}(r,{\mathbb C})$--oper
\begin{equation}\label{e47}
(P_{\rm PGL},\ \nabla, \ \{{\mathbf P}_j\}^{r-2}_{0})
\end{equation}
on the transversely complex foliation $(M,\, \widehat{\mathcal F})$. Let
\begin{equation}\label{e48}
p\ :\ {\mathbf P}_{r-1}\ :=\ P_{\rm PGL}\times^{{\rm PGL}(r,{\mathbb C})}{\mathbb C}{\mathbb P}^{r-1}
\ \longrightarrow\ M
\end{equation}
be the associated transversely holomorphic ${\mathbb C}{\mathbb P}^{r-1}$--bundle on transversely complex
foliation $(M,\, \widehat{\mathcal F})$. The flat connection $\nabla$
on $P_{\rm PGL}$ induces a flat connection
\begin{equation}\label{e48a}
D
\end{equation}
on ${\mathbf P}_{r-1}$ which is compatible with the transversely holomorphic structure. For any
$0\, \leq\, j\, \leq\, r-2$, let
\begin{equation}\label{e49}
p_j\ :\ {\mathbf P}_j\ \longrightarrow\ M
\end{equation}
be the natural projection. So $p_j$ is the restriction of $p$ in \eqref{e48} (recall that
${\mathbf P}_j\, \subset\, {\mathbf P}_{r-1}$), and $p_0$ is transversely biholomorphic.
Let
\begin{equation}\label{e50}
\sigma\ :=\ p^{-1}_0 \ :\ M\ \longrightarrow\ {\mathbf P}_{r-1}
\end{equation}
be the section of the projection $p$ in \eqref{e48}.

Let $T^p\, \subset\, T{\mathbf P}_{r-1}\otimes {\mathbb C}$ be the relative transversely 
holomorphic tangent bundle over ${\mathbf P}_{r-1}$ for the projection $p$. Construct 
the transversely holomorphic bundle
$$
{\mathcal J}^1_p (T^p)\ \longrightarrow\ {\mathbf P}_{r-1}
$$
(as in \eqref{d4}) which is a quotient of $J^1_p (T^p)$
\begin{equation}\label{e51a}
\delta\ :\ J^1_p (T^p)\ \longrightarrow\ {\mathcal J}^1_p (T^p).
\end{equation}
Let
\begin{equation}\label{e51}
\varphi\ :\ {\mathcal J}^1_p (T^p)\ \longrightarrow\ T^p
\end{equation}
be the projection induced by the natural projection $J^1_p (T^p)\,\longrightarrow\,T^p$.

Consider the pullback $\sigma^*{\mathcal J}^1_p (T^p)$ on $M$, where $\sigma$ is the section
in \eqref{e50}. The isomorphism in \eqref{d3} produces a transversely holomorphic isomorphism
\begin{equation}\label{e52}
B\ :\ {\mathbf P}_{r-1} \ \longrightarrow\ P(\sigma^*{\mathcal J}^1_p (T^p))
\ =\ \sigma^* P({\mathcal J}^1_p (T^p))
\end{equation}
of transversely holomorphic ${\mathbb C}{\mathbb P}^{r-1}$--bundles.

For $0\, \leq\, j\, \leq\, r-2$, consider the transversely holomorphic relative tangent bundle
\begin{equation}\label{tpj}
T^{p_j}\,\ \longrightarrow\,\ {\mathbf P}_j
\end{equation}
for the projection $p_j$ in \eqref{e49}. The vector bundle
$\sigma^*T^p\, \longrightarrow\, M$ (see \eqref{e51} and \eqref{e50}) has the following filtration of
subbundles:
$$
0\,=\, \sigma^* T^{p_0}\, \subset\, \sigma^* T^{p_1} \, \subset\, \cdots \, \subset\, 
\sigma^* T^{p_{r-2}} \, \subset\, \sigma^*T^p,
$$
where $T^{p_j}$ is constructed in \eqref{tpj}. Using the projection $\sigma^*\varphi\ :\ 
\sigma^*{\mathcal J}^1_p (T^p)\ \longrightarrow\ \sigma^*T^p$ (see \eqref{e51}), this 
gives the following filtration of subbundles of $\sigma^*{\mathcal J}^1_p (T^p)$:
\begin{equation}\label{e53}
0\ \subset\ (\sigma^*\varphi)^{-1}(\sigma^* T^{p_0})\ \subset\ (\sigma^*\varphi)^{-1}(\sigma^* T^{p_1}) \
\subset\ \cdots 
\end{equation}
$$
\subset\ (\sigma^*\varphi)^{-1}(\sigma^* T^{p_{r-2}})\ \subset\
(\sigma^*\varphi)^{-1}(\sigma^*T^p)\ =\ \sigma^* {\mathcal J}^1_p (T^p).
$$
For $0\, \leq\, j\, \leq\, r-2$, we have
\begin{equation}\label{e54}
B({\mathbf P}_{j})\,\ =\, \ P((\sigma^*\varphi)^{-1}(\sigma^* T^{p_j})),
\end{equation}
where $B$ is the isomorphism in \eqref{e52}. Indeed, \eqref{e54} follows from the condition in
Definition \ref{def3} that the homomorphism $\Phi_j$ in \eqref{e43} is an isomorphism for all
$0\, \leq\, j\, \leq\, r-2$.

Take another ${\rm PGL}(r,{\mathbb C})$ transversely oper
\begin{equation}\label{e55}
(P'_{\rm PGL},\ \nabla', \ \{{\mathbf P}'_j\}^{r-2}_{0})
\end{equation}
on the transversely complex foliation $(M,\, \widehat{\mathcal F})$. As in \eqref{e48}, let
\begin{equation}\label{e56}
p'\ :\ {\mathbf P}'_{r-1}\ :=\ P'_{\rm PGL}\times^{{\rm PGL}(r,{\mathbb C})}{\mathbb C}{\mathbb P}^{r-1}
\ \longrightarrow\ M
\end{equation}
be the associated transversely holomorphic ${\mathbb C}{\mathbb P}^{r-1}$--bundle on
the transversely complex foliation $(M,\, \widehat{\mathcal F})$.

\begin{theorem}\label{thm2}
There is a natural transversely holomorphic isomorphism
$$
\Psi\ :\ {\mathbf P}_{r-1}\ \longrightarrow\ {\mathbf P}'_{r-1}
$$
of transversely holomorphic ${\mathbb C}{\mathbb P}^{r-1}$--bundles such that
$$
\Psi({\mathbf P}_j)\,\ =\,\ {\mathbf P}'_j
$$
for all $0\, \leq \, j \,\leq\, r-2$.
\end{theorem}

\begin{proof}
Consider the normal bundle ${\mathbf N}^{r-1}_{r-2}\, \longrightarrow\, {\mathbf P}_{r-2}$
constructed in \eqref{e41}. From the condition in
Definition \ref{def3} that the homomorphism $\Phi_j$ in \eqref{e43} is an isomorphism for all
$0\, \leq\, j\, \leq\, r-2$ it follows that
\begin{equation}\label{e57}
\sigma^*{\mathbf N}^{r-1}_{r-2}\,\ =\,\ (N^{1,0})^{\otimes (r-1)},
\end{equation}
where $\sigma$ and $N^{1,0}$ are constructed in \eqref{e50} and \eqref{ej2} respectively.
More generally, we have $\sigma^*{\mathbf N}^j_{j-1}\,=\, (N^{1,0})^{\otimes j}$.

Consider the sheaf of transversely holomorphic sections of the transversely holomorphic vector bundle
$T^p$ on ${\mathbf P}_{r-1}$ (see \eqref{e51}); this sheaf will also be denoted by $T^p$.
Its direct image
\begin{equation}\label{e58}
p_*T^p\,\ \longrightarrow\, \ M
\end{equation}
is equipped with a flat connection induced by the flat connection $D$ on
${\mathbf P}_{r-1}$ (see \eqref{e48a}). This flat connection on $p_*T^p$, which
will be denoted by $\mathbb D$, is compatible with the transversely holomorphic structure
of $p_*T^p$.

Take a point $y\, \in\, M$, and also take a flat section $\lambda$ of $p_*T^p$ defined on
an open neighborhood $U_y\, \subset\, M$ of $y$.
So $\lambda$ gives a section of $T^p$ on $p^{-1}(U)$; denote this section of $T^p$ on $p^{-1}(U)$
by $\widetilde{\lambda}$. Restrict $\widetilde{\lambda}$ to $\sigma(U)\,\subset\,
{\mathbf P}_{r-1}$, where $\sigma$ is the section in \eqref{e50}; so we have
\begin{equation}\label{e58a}
\widetilde{\lambda}'\ :=\ \widetilde{\lambda}\big\vert_{\sigma(U)}\ :\
\sigma(U)\ \longrightarrow\ T^p\big\vert_{\sigma(U)}.
\end{equation}

Let
$$
\iota\ :\ {\mathbf P}_{r-2}\ \hookrightarrow\ {\mathbf P}_{r-1}
$$
be the inclusion map. We have a natural quotient map
$$
\rho_p\ :\ \iota^* T^p\ \longrightarrow\ {\mathbf N}^{r-1}_{r-2},
$$
where ${\mathbf N}^{r-1}_{r-2}$ is constructed in \eqref{e41}. Composing this
$\rho_p$ with $\widetilde{\lambda}'$ in \eqref{e58a}, we have
$$
\rho_p\circ \widetilde{\lambda}'\,\ :\,\ \sigma(U)\,\ \longrightarrow\,\
{\mathbf N}^{r-1}_{r-2}\big\vert_{\sigma(U)}.
$$
This map $\rho_p\circ\widetilde{\lambda}'$ and the isomorphism in \eqref{e57} together produce a
map
$$
\lambda_2\ :\ U\ \longrightarrow\ (N^{1,0})^{\otimes (r-1)}\big\vert_U.
$$
The above construction of $\lambda_2$ from $\lambda$ yields a homomorphism
\begin{equation}\label{e59}
\rho\ :\ p_*T^p\ \longrightarrow\ J^{r-1}_{\widehat{\mathcal F}}((N^{1,0})^{\otimes (r-1)});
\end{equation}
see \eqref{e25} for the transversal jet bundle $J^{r-1}_{\widehat{\mathcal F}}((N^{1,0})^{\otimes (r-1)})$.

As before, take a point $y\, \in\, M$, and also take a flat section ${\mathbf\eta}$ of $p_*T^p$ defined on
an open neighborhood $U_y\, \subset\, M$ of $y$. The section of $T^p$ over $p^{-1}(U)$ given by
${\mathbf\eta}$ will be denoted by $\widetilde{\mathbf\eta}$.
Restricting $\widetilde{\mathbf\eta}$ to the first order infinitesimal neighborhood of $\sigma(z)$,
where $\sigma$ is the section in \eqref{e50}, 
we get a homomorphism $p_*T^p\, \longrightarrow\, \sigma^* J^1_p (T^p)$, and composing it with
$\delta$ in \eqref{e51a} a homomorphism
$$
\widehat{\delta}\ :\ p_*T^p \ \longrightarrow\ \sigma^* {\mathcal J}^1_p (T^p)
$$
is obtained. The kernel of this homomorphism $\widehat{\delta}$ coincides with the kernel of
$\rho$ in \eqref{e59}. Consider the identity map of $p_*T^p$; using
$\widehat{\delta}$ and $\delta$, it produces a
transversely holomorphic isomorphism
\begin{equation}\label{e60}
\widehat{\textbf{C}}\ :\ J^{r-1}_{\widehat{\mathcal F}}((N^{1,0})^{\otimes (r-1)})
\ \stackrel{\sim}\longrightarrow\ \sigma^*{\mathcal J}^1_p (T^p).
\end{equation}

The transversely holomorphic vector bundle $J^{r-1}_{\widehat{\mathcal F}}((N^{1,0})^{\otimes (r-1)})$
has the following filtration of transversely holomorphic subbundles given by the lower order 
transversal jet bundles:
\begin{equation}\label{e61}
J^{r-1}_{\widehat{\mathcal F}}((N^{1,0})^{\otimes (r-1)})\, \longrightarrow\,
J^{r-2}_{\widehat{\mathcal F}}((N^{1,0})^{\otimes (r-1)}) \, \longrightarrow\, \cdots
\, \longrightarrow\, J^{2}_{\widehat{\mathcal F}}((N^{1,0})^{\otimes (r-1)})
\end{equation}
$$
\, \longrightarrow\, J^{1}_{\widehat{\mathcal F}}((N^{1,0})^{\otimes (r-1)})
\, \longrightarrow\, J^{0}_{\widehat{\mathcal F}}((N^{1,0})^{\otimes (r-1)})\,=\,
(N^{1,0})^{\otimes (r-1)} \, \longrightarrow\, 0.
$$
The isomorphism $\widehat{\textbf{C}}$ in \eqref{e60} takes the filtration of
$J^{r-1}_{\widehat{\mathcal F}}((N^{1,0})^{\otimes (r-1)})$ in \eqref{e61} to the
filtration of $\sigma^*{\mathcal J}^1_p (T^p)$ in \eqref{e53}.

Let
\begin{equation}\label{e62}
\textbf{C}\ :\ P(J^{r-1}_{\widehat{\mathcal F}}((N^{1,0})^{\otimes (r-1)}))
\ \stackrel{\sim}\longrightarrow\ P(\sigma^*{\mathcal J}^1_p (T^p))\ =\
\sigma^* P({\mathcal J}^1_p (T^p))
\end{equation}
be the transversely holomorphic isomorphism of transversely holomorphic ${\mathbb C}
{\mathbb P}^{r-1}$--bundles given by $\widehat{\textbf{C}}$ in \eqref{e60}.

Combining the isomorphism $\textbf{C}$ in \eqref{e62} with the isomorphism $B$ in \eqref{e52},
we get a transversely holomorphic isomorphism
\begin{equation}\label{e63}
\textbf{C}^{-1}\circ B\ :\ {\mathbf P}_{r-1} \ \longrightarrow\ 
P(J^{r-1}_{\widehat{\mathcal F}}((N^{1,0})^{\otimes (r-1)}))
\end{equation}
of transversely holomorphic ${\mathbb C}{\mathbb P}^{r-1}$--bundles.

Since the isomorphism $\widehat{\textbf{C}}$ in \eqref{e60} takes the filtration of
$J^{r-1}_{\widehat{\mathcal F}}((N^{1,0})^{\otimes (r-1)})$ in \eqref{e61} to the
filtration of $\sigma^*{\mathcal J}^1_p (T^p)$ in \eqref{e53}, using \eqref{e54} we conclude that
the isomorphism $\textbf{C}^{-1}\circ B$ in \eqref{e63} takes the filtration of projective subbundles
$\{{\mathbf P}_j\}^{r-2}_{0}$ of ${\mathbf P}_{r-1}$ to the filtration of projective subbundles of
$P(J^{r-1}_{\widehat{\mathcal F}}((N^{1,0})^{\otimes (r-1)}))$ given by the filtration of
vector bundles in \eqref{e61}.

Therefore, the projective bundle ${\mathbf P}_{r-1}$ and its filtration of subbundles
$\{{\mathbf P}_j\}^{r-2}_{0}$ is identified with a filtered projective bundle which does not
depend on the oper. This completes the proof.
\end{proof}

\section*{Declaration}

No data were used or generated in this project. The authors do not have any conflict of 
interests.


\end{document}